\documentclass{amsart}

\usepackage{amsmath}
\usepackage{amssymb}
\usepackage{bbm}
\usepackage{enumerate}
\usepackage{appendix}
\usepackage{mathrsfs}

\numberwithin{equation}{section}

\newcommand{\R}{\mathbb{R}}
\newcommand{\N}{\mathbb{N}}

\newcommand{\E}{\mathbb{E}}
\newcommand{\pr}{\mathbb{P}}

\newcommand{\e}{\operatorname{e}}
\newcommand{\dd}{\,{\mathrm d}}
\newcommand{\db}{{\mathrm d}}

\newcommand{\im}{\operatorname{i}}
\newcommand{\ind}{\mathbbm{1}}

\newcommand{\lo}{\mathcal{L}}

\newcommand{\eps}{\varepsilon}

\newcommand{\pt}{\partial}

\newcommand{\std}{\,\partial}
\newcommand{\stb}{\partial}

\newcommand{\scale}{\operatorname{\delta}}
\newcommand{\scalemat}{\operatorname{\delta}}
\newcommand{\mineig}{\lambda_{\rm min}}
\newcommand{\fil}{\mathcal{C}}
\newcommand{\bracol}{\mathcal{A}}
\newcommand{\ord}{\mathcal{O}}

\newtheorem{lemma}{Lemma}[section]

\newtheorem{propn}[lemma]{Proposition}
\newtheorem{thm}[lemma]{Theorem}
\newtheorem{cor}[lemma]{Corollary}
\newtheorem{defn}[lemma]{Definition}
\newtheorem{remark0}[lemma]{Remark}
\newtheorem{eg0}[lemma]{Example}

\newenvironment{remark}{\begin{remark0}\rm}{\hspace*{\fill} $\square$
                        \end{remark0}}

\author[K. Habermann]{Karen Habermann}
\address{Statslab, Centre for Mathematical Sciences,
Wilberforce Road,
Cambridge,
CB3 0WB,
United Kingdom.}
\email{K.Habermann@maths.cam.ac.uk}
\thanks{Statistical Laboratory, University of Cambridge. Research
  supported by EPSRC grant EP/H023348/1 for the Cambridge Centre for
  Analysis}

\title[Small-time fluctuations for sub-Riemannian diffusion loops]
{Small-time fluctuations for sub-Riemannian diffusion loops}
\begin{document}
\begin{abstract}
  We study the small-time fluctuations for diffusion processes which
  are conditioned by their initial and final positions, under the
  assumptions that the diffusivity has a sub-Riemannian structure and
  that the drift vector field lies in the span of the
  sub-Riemannian structure.
  In the case where the
  endpoints agree and the generator of the diffusion process
  is non-elliptic at that point, the
  deterministic Malliavin covariance matrix is always degenerate.
  We identify, after a suitable rescaling, another limiting
  Malliavin covariance matrix which is non-degenerate, and we show
  that, with the same scaling, the diffusion
  Malliavin covariance matrices are uniformly non-degenerate.
  We further show that the suitably rescaled fluctuations of the
  diffusion loop converge to a
  limiting diffusion loop, which is equal in law to
  the loop we obtain by
  taking the limiting process of the unconditioned rescaled
  diffusion processes and condition it to return to its starting point.
  The generator of the unconditioned limiting rescaled
  diffusion process can be described in terms of the
  original generator.
\end{abstract}
\maketitle
\thispagestyle{empty}

\section{Introduction}
The small-time asymptotics of heat kernels have been
extensively studied over the years, from an analytic, a geometric
as well as a probabilistic point of view.
Bismut \cite{bismut} used Malliavin calculus to perform the analysis
of the heat kernel
in the elliptic case and he developed a deterministic Malliavin
calculus to study
H\"ormander-type hypoelliptic heat
kernels. Following this approach, Ben Arous~\cite{GBAcut} found the
corresponding small-time asymptotics outside the sub-Riemannian cut
locus and Ben Arous \cite{GBAdiag} and L\'eandre \cite{leandre}
studied the behaviour on the diagonal. In joint
work~\cite{expdecayI},~\cite{expdecayII},
they also discussed the exponential decay of
hypoelliptic heat kernels on the diagonal.

In recent years, there has been further progress in the study of
heat kernels on sub-Riemannian manifolds. Barilari, Boscain and Neel
\cite{boscain} found estimates of the heat kernel on the cut locus by
using an analytic approach, and Inahama and Taniguchi
\cite{yuzuru} combined
Malliavin calculus and rough path theory to determine small-time full
asymptotic expansions on the off-diagonal cut locus. Moreover,
Bailleul, Mesnager and Norris \cite{BMN} studied the asymptotics of
sub-Riemannian diffusion bridges outside the cut
locus. We extend their analysis to the diagonal and
describe the asymptotics of sub-Riemannian diffusion loops.
In a suitable chart, and after
a suitable rescaling, we show that the small-time diffusion loop measures 
have a non-degenerate limit, which we identify explicitly in terms of a certain
local limiting operator. Our analysis also allows us to determine the 
loop asymptotics under the scaling used to obtain a small-time
Gaussian limit of the sub-Riemannian diffusion bridge measures
in \cite{BMN}. In general, these asymptotics
are now degenerate and need no longer be Gaussian.

Let $M$ be a connected smooth manifold of dimension $d$ and
let $a$ be a smooth non-negative quadratic form on the cotangent
bundle $T^*M$. Let $\lo$ be a second order differential operator on
$M$ with smooth coefficients, such that $\lo 1=0$
and such that $\lo$ has principal symbol $a/2$. One refers to $a$ as the
diffusivity of the operator $\lo$. We say that $a$ has a
sub-Riemannian structure if there exist $m\in\N$ and smooth vector
fields $X_1,\dots,X_m$ on $M$ satisfying the strong H\"{o}rmander
condition,
i.e. the vector fields together with their commutator brackets of all
orders span $T_yM$ for all $y\in M$,
such that
\begin{equation*}
  a(\xi,\xi)=\sum_{i=1}^m\langle\xi, X_i(y)\rangle^2
  \quad\mbox{for}\quad \xi\in T_y^* M\;.
\end{equation*}
Thus, we can write
\begin{equation*}
  \lo=\frac{1}{2}\sum_{i=1}^m X_i^2+X_0
\end{equation*}
for a vector field $X_0$ on $M$, which we also assume to be smooth.
Note that the vector fields $X_0,X_1,\dots,X_m$ are allowed to vanish
and hence, the sub-Riemannian structure $(X_1,\dots,X_m)$ need not be
of constant rank.
To begin with, we impose the global condition
\begin{equation}\label{globalcond}
  M=\R^d\quad\mbox{and}\quad X_0,X_1,\dots,X_m\in C_b^\infty(\R^d,\R^d)\;,
\end{equation}
subject to the additional constraint
$X_0(y)\in\operatorname{span}\{X_1(y),\dots,X_m(y)\}$
for all $y\in\R^d$. Subsequently, we follow Bailleul, Mesnager and
Norris \cite{BMN} and insist that there exist a
smooth one-form $\beta$ on $M$ with
$\|a(\beta,\beta)\|_\infty<\infty$,
and a locally invariant positive smooth measure $\tilde\nu$ on $M$ such that,
for all $f\in C^\infty(M)$,
\begin{equation*}
  \lo f =\frac{1}{2}\operatorname{div}(a\nabla f) + a(\beta,\nabla f)\;.
\end{equation*}
Here the divergence is understood with respect to
$\tilde\nu$. Note that if
the operator $\lo$ is of this form then
$X_0=\sum_{i=1}^m\alpha_i X_i$ with
$\alpha_i=\frac{1}{2}\operatorname{div} X_i+\beta(X_i)$
and in particular, $X_0(y)\in\operatorname{span}\{X_1(y),\dots,X_m(y)\}$
for all $y\in M$.

We are interested in the
associated diffusion bridge measures.
Fix $x\in M$ and let $\eps>0$. If we do not assume the global condition then
the diffusion process $(x_t^\eps)_{t<\zeta}$
defined up to the explosion time~$\zeta$
starting from
$x$ and having generator $\eps\lo$ may explode with positive
probability before time~$1$. Though,
on the event $\{\zeta>1\}$, the process
$(x_t^\eps)_{t\in[0,1]}$ has a unique sub-probability law
$\mu_\eps^x$ on  
the set of continuous paths $\Omega=C([0,1],M)$.
Choose a positive smooth measure $\nu$ on $M$,
which can differ from the
locally invariant positive measure $\tilde\nu$ on $M$, and let $p$
denote the Dirichlet heat kernel for $\lo$ with respect to $\nu$.
We can disintegrate $\mu_\eps^x$ to give a unique family of probability
measures $(\mu_\eps^{x,y}\colon y\in M)$ on $\Omega$ such that
\begin{equation*}
  \mu_\eps^x(\db\omega)=\int_M\mu_\eps^{x,y}(\db\omega)
  p(\eps,x,y)\nu(\db y)\;,
\end{equation*}
with $\mu_\eps^{x,y}$ supported on
$\Omega^{x,y}=\{\omega\in\Omega\colon\omega_0=x,\omega_1=y\}$ for all
$y\in M$ and where the map $y\mapsto\mu_\eps^{x,y}$ is weakly continuous.
Bailleul, Mesnager and Norris \cite{BMN} studied the
small-time fluctuations of the
diffusion bridge measures $\mu_\eps^{x,y}$ in the limit $\eps\to 0$
under the assumption that $(x,y)$ lies outside the sub-Riemannian cut
locus. Due to the latter condition, their results do not cover the
diagonal case unless $\lo$ is elliptic at $x$.
We show how to extend their analysis in order
to understand the small-time fluctuations of the
diffusion loop measures $\mu_\eps^{x,x}$.

As a by-product, we recover
the small-time heat kernel asymptotics
on the diagonal shown by Ben Arous \cite{GBAdiag} and
L\'eandre \cite{leandre}.
Even though our approach for obtaining the small-time asymptotics on the
diagonal is similar to \cite{GBAdiag}, it does not rely on the Rothschild
and Stein lifting theorem, cf.~\cite{stein}. Instead, we use the
notion of an adapted chart at $x$, introduced by Bianchini and Stefani
\cite{bianchini}, which provides suitable coordinates around $x$.
We discuss adapted charts in detail in Section~\ref{sec:graded}.
The
chart Ben Arous \cite{GBAdiag} performed his analysis in is in fact
one specific
example of an adapted chart, whereas we allow for any adapted
chart. In the case where the diffusivity $a$ has a sub-Riemannian
structure which is one-step bracket-generating at $x$,
any chart around
$x$ is adapted. However, in general these charts are more complex and for
instance, even if $M=\R^d$
there is no reason to assume that the identity
map
is adapted. Paoli~\cite{paoli} successfully used adapted
charts to describe the small-time asymptotics of H\"ormander-type
hypoelliptic operators with non-vanishing
drift at a stationary point of the drift field.

To a sub-Riemannian structure $(X_1,\dots,X_m)$ on $M$, we associate a
linear scaling map $\scale_\eps\colon\R^d\to\R^d$ in a suitable set of
coordinates, which depends on the number of brackets needed to achieve
each direction,
and the so-called nilpotent approximations
$\tilde X_1,\dots,\tilde X_m$, which are homogeneous vector fields on
$\R^d$. For the details see Section~\ref{sec:graded}. The map
$\scale_\eps$ allows us to rescale the fluctuations of the diffusion
loop to high enough orders so
as to obtain a non-degenerate limiting measure, and the nilpotent
approximations are used to describe this limiting measure.
Let $(U,\theta)$ be an adapted chart around $x\in M$.
Smoothly extending this chart to all of $M$ yields a
smooth map $\theta\colon M\to\R^d$ whose derivative
$\db\theta_x\colon T_xM\to\R^d$ at $x$ is invertible.
Write $T\Omega^{0,0}$  for the set of continuous paths
$v=(v_t)_{t\in[0,1]}$ in $T_xM$ with $v_0=v_1=0$. Define
a rescaling map
$\sigma_\eps\colon\Omega^{x,x}\to T\Omega^{0,0}$ by
\begin{equation*}
  \sigma_\eps(\omega)_t=(\db\theta_x)^{-1}
  \left(\scale_\eps^{-1}\left(\theta(\omega_t)-\theta(x)\right)\right)
\end{equation*}
and let $\tilde\mu_\eps^{x,x}$ be the pushforward measure of
$\mu_\eps^{x,x}$ by $\sigma_\eps$, i.e. $\tilde\mu_\eps^{x,x}$ is the
unique probability measure on $T\Omega^{0,0}$ given by
\begin{equation*}
  \tilde\mu_\eps^{x,x}=\mu_\eps^{x,x}\circ\sigma_\eps^{-1}\;.
\end{equation*}
Our main result concerns the weak convergence of these rescaled
diffusion loop measures $\tilde\mu_\eps^{x,x}$.
To describe the limit,
assuming that $\theta(x)=0$, we consider the
diffusion process $(\tilde x_t)_{t\geq 0}$ in $\R^d$
starting from $0$ and having generator
\begin{equation*}
  \tilde\lo=\frac{1}{2}\sum_{i=1}^m\tilde X_i^2\;.
\end{equation*}
A nice cascade structure of the nilpotent
approximations $\tilde X_1,\dots,\tilde X_m$
ensures that this process
exists for all time.
Let $\tilde\mu^{0,\R^d}$ denote the law of the
diffusion
process $(\tilde x_t)_{t\in[0,1]}$ on the set of
continuous paths $\Omega(\R^d)=C([0,1],\R^d)$. By disintegrating
$\tilde\mu^{0,\R^d}$, we obtain the loop measure $\tilde\mu^{0,0,\R^d}$
supported on the set
$\Omega(\R^d)^{0,0}=
\{\omega\in\Omega(\R^d)\colon \omega_0=\omega_1=0\}$.
Define a map
$\rho\colon\Omega(\R^d)^{0,0}\to T\Omega^{0,0}$ by
\begin{equation*}
  \rho(\omega)_t=(\db\theta_x)^{-1}\omega_t
\end{equation*}
and set $\tilde\mu^{x,x}=\tilde\mu^{0,0,\R^d}\circ\rho^{-1}$. This is
the desired limiting probability measure on $T\Omega^{0,0}$.
\begin{thm}[Convergence of the rescaled diffusion bridge measures]
  \label{thm:convofbridges}
  Let $M$ be a connected smooth manifold and fix $x\in M$.
  Let $\lo$ be a second order
  partial differential operator on $M$ such that,
  for all $f\in C^\infty(M)$,
  \begin{equation*}
    \lo f=\frac{1}{2}\operatorname{div}(a\nabla f)+a(\beta,\nabla f)\;,
  \end{equation*}
  with respect to a locally invariant positive smooth measure,
  and where the smooth non-negative quadratic form $a$ on $T^*M$ has a
  sub-Riemannian structure and the smooth one-form $\beta$ on $M$
  satisfies $\|a(\beta,\beta)\|_\infty<\infty$.
  Then the rescaled diffusion loop measures
  $\tilde\mu_\eps^{x,x}$ converge weakly to the probability measure
  $\tilde\mu^{x,x}$ on $T\Omega^{0,0}$ as $\eps\to 0$.
\end{thm}
We prove this result by localising
Theorem~\ref{propn:convofbridges}.
As a consequence of the localisation argument,
Theorem~\ref{thm:convofbridges} remains
true under the weaker assumption that the smooth vector fields
giving the sub-Riemannian structure are only locally defined.
The theorem below
imposes an additional constraint on the map $\theta$ which ensures
that we can rely on the tools of Malliavin calculus to prove it.
As we
see later, the existence of such a diffeomorphism
$\theta$ is always guaranteed.
\begin{thm}
  \label{propn:convofbridges}
  Fix $x\in\R^d$. Let $X_0,X_1,\dots,X_m$ be smooth bounded
  vector fields on
  $\R^d$, with bounded derivatives of all orders, which satisfy the
  strong H\"ormander condition everywhere
  and suppose that $X_0(y)\in\operatorname{span}
  \{X_1(y),\dots,X_m(y)\}$ for all
  $y\in\R^d$. Set
  \begin{equation*}
    \lo=\frac{1}{2}\sum_{i=1}^m X_i^2+X_0\;.
  \end{equation*}
  Assume that the smooth map $\theta\colon\R^d\to\R^d$ is a global
  diffeomorphism with bounded derivatives of all positive
  orders and an adapted chart at $x$.
  Then the rescaled diffusion loop measures
  $\tilde\mu_\eps^{x,x}$ converge weakly to the probability measure
  $\tilde\mu^{x,x}$ on $T\Omega^{0,0}$ as $\eps\to 0$.
\end{thm}
Note that the limiting measures with respect to two different
choices of admissible diffeomorphisms $\theta_1$ and $\theta_2$
are related by the Jacobian matrix of the transition map
$\theta_2\circ\theta_1^{-1}$.

The proof of Theorem~\ref{propn:convofbridges} follows
\cite{BMN}. The additional technical result needed in our analysis
is the uniform non-degeneracy of the $\scale_\eps$-rescaled Malliavin
covariance matrices. Throughout the paper, we consider
Malliavin covariance matrices in the sense of Bismut and refer to what
is also called the reduced Malliavin covariance matrix
simply as the Malliavin covariance matrix.
Under the global assumption, there exists a
unique diffusion process  $(x_t^\eps)_{t\in[0,1]}$ starting at $x$ and
having generator~$\eps\lo$. Choose $\theta\colon\R^d\to\R^d$ as in
Theorem~\ref{propn:convofbridges} and define
$(\tilde x_t^\eps)_{t\in[0,1]}$
to be the rescaled diffusion process given by
\begin{equation*}
  \tilde x_t^\eps=
  \scale_\eps^{-1}\left(\theta(x_t^\eps)-\theta(x)\right)\;.
\end{equation*}
Denote the Malliavin covariance matrix of $\tilde x_1^\eps$ by
$\tilde c_1^\eps$. We know that, for each $\eps>0$, the matrix
$\tilde c_1^\eps$ is non-degenerate because the vector fields
$X_1,\dots,X_m$ satisfy the strong H\"ormander condition everywhere.
We prove that
these Malliavin covariance matrices
are in fact uniformly non-degenerate.
\begin{thm}
[Uniform non-degeneracy of the
$\scale_\eps$-rescaled Malliavin covariance matrices]
  \label{thm:UND}
  Let $X_0,X_1,\dots,X_m$ be smooth bounded vector fields on $\R^d$, with
  bounded derivatives of all orders, which satisfy the
  strong H\"ormander condition everywhere and such that
  $X_0(y)\in\operatorname{span}\{X_1(y),\dots,X_m(y)\}$ for all
  $y\in\R^d$. Fix $x\in\R^d$ and consider the diffusion operator
  \begin{equation*}
    \lo=\frac{1}{2}\sum_{i=1}^m X_i^2+X_0\;.
  \end{equation*}
  Then the rescaled Malliavin
  covariance matrices $\tilde c_1^\eps$ are uniformly non-degenerate,
  i.e. for all $p<\infty$, we have
  \begin{equation*}
    \sup_{\eps\in(0,1]}\E\left[\left|\det\left(\tilde c_1^\eps\right)^{-1}
      \right|^p\right]<\infty\;.
  \end{equation*}
\end{thm}
We see that
the uniform non-degeneracy of the rescaled Malliavin covariance
matrices $\tilde c_1^\eps$ is a consequence of the non-degeneracy of the
limiting diffusion
process $(\tilde x_t)_{t\in[0,1]}$
with generator $\tilde\lo$. The latter is implied by the
nilpotent approximations $\tilde X_1,\dots,\tilde X_m$ 
satisfying the strong H\"ormander condition everywhere on $\R^d$,
as proven in
Section~\ref{sec:graded}.

\medskip

\paragraph{{\it Organisation of the paper}}
The paper is organised as follows.
In Section~\ref{sec:graded},
we define the scaling operator $\scale_\eps$ with which we rescale the
fluctuations of the diffusion loop to obtain a non-degenerate limit.
It also sets up notations for subsequent sections and proves
preliminary results from which we deduce properties of the
limiting measure. In Section~\ref{sec:UND}, we characterise the
leading-order terms of the rescaled Malliavin covariance matrices
$\tilde c_1^\eps$ as $\eps\to 0$
and use this to prove Theorem~\ref{thm:UND}.
Equipped with the uniform non-degeneracy result,
in Section~\ref{sec:convofdiffmeas}, we adapt the
analysis from~\cite{BMN} to prove
Theorem~\ref{propn:convofbridges}. The approach presented
is based on ideas from Azencott, Bismut and Ben Arous and
relies on tools from Malliavin calculus. Finally, in
Section~\ref{pp:local}, we employ a
localisation argument to prove Theorem~\ref{thm:convofbridges}
and give an example to illustrate the result. Moreover,
we discuss the
occurrence of non-Gaussian behaviour in the $\sqrt{\eps}$-rescaled
fluctuations of diffusion loops. 

\medskip

\paragraph{{\it Acknowledgement}}
I would like to thank James Norris
for suggesting this problem and for his guidance
and many helpful discussions.

\section{Graded structure and nilpotent approximation}
\label{sec:graded}
We introduce the notion of an adapted chart and of an associated
dilation $\delta_\eps\colon \R^d\to\R^d$ which allows us to rescale the
fluctuations of a diffusion loop in a way which gives rise to a
non-degenerate limit as $\eps\to 0$. To be able to
characterise this limiting
measure later, we define the nilpotent approximation
of a vector field on $M$
and show that the nilpotent approximations of a sub-Riemannian
structure form a sub-Riemannian structure themselves.
This section
is based on Bianchini and Stefani \cite{bianchini} and Paoli
\cite{paoli}, but we made some adjustments because the drift term $X_0$
plays a different role in our setting.
At the end, we
present an example to illustrate the various constructions.

\subsection{Graded structure induced by a
sub-Riemannian structure}
Let $(X_1,\dots,X_m)$ be a sub-Riemannian structure on $M$ and
fix $x\in M$. For $k\geq 1$, set
\begin{equation*}
  \bracol_k=
  \left\{[X_{i_1},[X_{i_2},\dots, [X_{i_{k-1}},X_{i_k}]\dots]]
  \colon 1\leq i_1,\dots,i_k\leq m\right\}
\end{equation*}
and, for $n\geq 0$,
define a subspace of the space of smooth vector fields
on $M$ by
\begin{equation*}
  C_n=\operatorname{span}\bigcup_{k=1}^n\bracol_k\;,
\end{equation*}
where the linear combinations are taken over $\R$.
Note that $C_0=\{0\}$. Let
$C=\operatorname{Lie}\{X_1,\dots,X_m\}$ be the Lie algebra over $\R$
generated by the vector fields $X_1,\dots,X_m$. We observe that
$C_n\subset C_{n+1}$ as well as
$[C_{n_1},C_{n_2}]\subset C_{n_1+n_2}$ for $n_1,n_2\geq 0$ and that
$\bigcup_{n\geq 0} C_n=C$. Hence, $\fil=\{C_n\}_{n\geq 0}$ is an
increasing filtration of the subalgebra $C$ of the Lie algebra of
smooth vector fields on $M$.
Consider the subspace $C_n(x)$ of the tangent space $T_xM$ given by
\begin{equation*}
  C_n(x)=\left\{X(x)\colon X\in C_n\right\}\;.
\end{equation*}
Define $d_n=\operatorname{dim} C_n(x)$.
Since
$X_1,\dots,X_m$ are assumed to satisfy the strong H\"ormander
condition, we have $\bigcup_{n\geq 0} C_n(x) = T_x M$, and it follows
that
\begin{equation*}
  N=\min\{n\geq 1\colon d_n=d\}
\end{equation*}
is well-defined. We call $N$ the step of the filtration $\fil$ at $x$.
\begin{defn}
  A chart $(U,\theta)$ around $x\in M$ is called an adapted
  chart to the filtration $\fil$ at $x$ if $\theta(x)=0$ and,
  for all $n\in\{1,\dots,N\}$,
  \begin{enumerate}[{\rm (i)}]
  \item \label{cond1}$\displaystyle C_n(x)=
    \operatorname{span}\left\{\frac{\pt}{\pt\theta^1}(x), 
      \dots,\frac{\pt}{\pt\theta^{d_n}}(x)\right\}\;,$ and
  \item \label{cond2}
    $\left(\operatorname{D}\theta^k\right)(x)=0$ for every
    differential operator $\operatorname{D}$ of the form
    \begin{equation*}
      \operatorname{D}=Y_1\dots Y_n\quad\mbox{with}\quad
      Y_1,\dots,Y_n\in\{X_1,\dots,X_m\}
    \end{equation*}
    and all $k>d_n\;.$
  \end{enumerate}
\end{defn}
Note that condition (\ref{cond2}) is equivalent to requiring that
$(\operatorname{D}\theta^k)(x)=0$ for every
differential operator
$\operatorname{D}\in\operatorname{span}
\{Y_1\cdots Y_j\colon Y_l\in C_{i_l}\mbox{ and }
i_1+\dots+i_j\leq n\}$ and all $k>d_n$.
The existence of an adapted chart to the filtration $\fil$ at $x$ is
ensured by \cite[Corollary~3.1]{bianchini}, which explicitly
constructs such a chart by considering the integral curves of the
vector fields $X_1,\dots,X_m$. However, we should keep in mind that
even though being adapted at $x$ is a local property,
the germs of adapted charts at $x$ need not coincide.

Unlike Bianchini and Stefani \cite{bianchini}, we choose to
construct 
our graded structure on $\R^d$ instead of on the domain $U$ of an
adapted chart, as this works better with our analysis.
Define weights $w_1,\dots,w_d$ by setting
$w_k=\min\{l\geq 1\colon d_l\geq k\}$.
This definition immediately implies
$1\leq w_1 \leq \dots\leq w_d=N$.
Let $\delta_\eps\colon\R^d\to\R^d$ be the anisotropic dilation
given by
\begin{equation*}
  \delta_\eps(y)=\delta_\eps\left(y^1,\dots,y^k,\dots,y^d\right)=
  \left(\eps^{w_1/2}y^1,\dots,\eps^{w_k/2}y^k,\dots,\eps^{w_d/2}y^d\right)\;,
\end{equation*}
where $(y^1,\dots,y^d)$ are Cartesian coordinates on $\R^d$.
For a non-negative integer $w$,
a polynomial $g$ on $\R^d$ is called homogeneous of weight $w$
if it satisfies $g\circ\delta_\eps=\eps^{w/2}g$.
For instance, the monomial $y_1^{\alpha_1}\dots y_d^{\alpha_d}$ is homogeneous
of weight $\sum_{k=1}^d\alpha_k w_k$. We denote the set of
polynomials which are homogeneous
of weight $w$ by $\mathcal{P}(w)$. Note that
the zero polynomial is
contained in $\mathcal{P}(w)$ for all non-negative integers $w$.
Following \cite{bianchini}, the graded order $\ord(g)$ of a polynomial
$g$ is defined by the property
\begin{equation*}
  \ord(g)\geq i \quad\mbox{if and only if}\quad
  g\in\bigoplus_{w\geq i}\mathcal{P}(w)\;.
\end{equation*}
Thus, the graded order of a non-zero polynomial $g$ is the maximal
non-negative integer $i$ such that
$g\in\oplus_{w\geq i}\mathcal{P}(w)$ whereas the graded order of the
zero polynomial is set to be $\infty$.
Similarly, for a smooth function $f\in C^\infty(V)$,
where $V\subset \R^d$ is an open neighbourhood of $0$, we define its
graded order $\ord(f)$ by requiring that $\ord(f)\geq i$ if
and only if each
Taylor approximation of $f$ at $0$ has graded order at least $i$.
We see
that the graded order of a smooth function is either a non-negative
integer or $\infty$.
Furthermore,
for an integer $a$,
a polynomial vector field $Y$ on $\R^d$ is called
homogeneous of weight $a$ if, for all $g\in\mathcal{P}(w)$, we have
$Y g\in\mathcal{P}(w-a)$. Here we set $\mathcal{P}(b)=\{0\}$
for negative integers $b$.
The weight of a general polynomial
vector field is defined to be the smallest weight of its homogeneous
components.
Moreover, the graded order $\ord(\operatorname{D})$ of a differential
operator $\operatorname{D}$ on $V$ is given by
saying that
\begin{equation*}
  \ord(\operatorname{D})\leq i\quad\mbox{if and only if}\quad
  \ord(\operatorname{D}g)\geq\ord(g)-i\mbox{ for all polynomials }g\;.
\end{equation*}
For example, the polynomial vector field
$y^1\frac{\pt}{\pt y^1}+(y^1)^2\frac{\pt}{\pt y^1}$
on $\R^d$ has weight $-w_1$
but considered as a differential operator it has graded order $0$.
It also follows that the graded order of a differential operator
takes values in $\mathbb{Z}\cup\{\pm\infty\}$ and that
the zero differential operator has graded order
$-\infty$.
In the remainder, we need the notions of the weight of a polynomial
vector field and the graded order of a vector field understood as a
differential operator.
For smooth vector fields $X_1$ and $X_2$ on $V$, it
holds true that
\begin{equation}\label{eq:ordineq}
  \ord([X_1,X_2])\leq \ord(X_1)+\ord(X_2)\;.
\end{equation}
We further observe that for any smooth vector field $X$ on $V$
and every integer $n$, there exists
a unique polynomial vector field $X^{(n)}$ of weight at least
$n$ such that
$\ord(X-X^{(n)})\leq n-1$, namely the sum of the homogeneous vector
fields of weight greater than or equal to $n$ in the formal Taylor
series of $X$ at $0$.
\begin{defn}
  Let $X$ be a smooth vector field on $V$.
  We call $X^{(n)}$ the graded
  approximation of weight $n$ of $X$.
\end{defn}
Note that $X^{(n)}$ is a polynomial vector field and hence, it can
be considered as a vector field defined on all of $\R^d$.

\subsection{Nilpotent approximation}
Let $(U,\theta)$ be an adapted chart to the filtration
induced by a sub-Riemannian structure $(X_1,\dots,X_m)$ on $M$
at $x$ and set
$V=\theta(U)$. Note that, for $i\in\{1,\dots,m\}$, the pushforward vector
field $\theta_*X_i$ is a vector field on $V$ and write $\tilde
X_i$ for the graded approximation $(\theta_*X_i)^{(1)}$ of weight $1$
of $\theta_*X_i$.
\begin{defn}
  The polynomial vector fields $\tilde X_1,\dots,\tilde X_m$ on $\R^d$
  are called the nilpotent approximations of the vector fields
  $X_1,\dots,X_m$ on $M$.
\end{defn}
By \cite[Theorem~3.1]{bianchini}, we know that
$\ord(\theta_*X_i)\leq 1$. Thus, the formal Taylor series of
$\theta_*X_i$ at $0$ cannot contain any homogeneous components of
weight greater than or equal to two.
This implies that $\tilde X_i$ is a homogeneous vector
field of weight $1$ and therefore,
\begin{equation*}
  \left(\delta_\eps^{-1}\right)_* \tilde X_i =\eps^{-1/2}\tilde X_i
  \quad\mbox{ for all }i\in\{1,\dots,m\}\;.
\end{equation*}
Moreover, from $\ord(\theta_*X_i-\tilde X_i)\leq 0$, we deduce that
\begin{equation*}
  \sqrt{\eps}\left(\delta_\eps^{-1}\right)_*
  (\theta_* X_i )\to\tilde X_i
  \quad\mbox{as}\quad\eps\to 0
  \quad\mbox{ for all }i\in\{1,\dots,m\}\;.
\end{equation*}
This convergence holds on all of $\R^d$ because
for $y\in\R^d$ fixed, we have $\delta_\eps(y)\in V$
for $\eps>0$ sufficiently small.
\begin{remark}\label{cascade}
  The vector fields
  $\tilde X_1,\dots,\tilde X_m$ on $\R^d$ have
  a nice cascade structure. Since $\tilde X_i$, for
  $i\in\{1,\dots,m\}$, contains the terms of
  weight $1$ the component $\tilde X_i^k$, for $k\in\{1,\dots,d\}$, does
  not depend on the coordinates with weight greater than or equal to
  $w_k$ and depends only linearly on the coordinates with weight $w_k-1$.
\end{remark}
We show that the nilpotent approximations
$\tilde X_1,\dots,\tilde X_m$
inherit the strong H\"ormander property from the sub-Riemannian structure
$(X_1,\dots,X_m)$. This result plays a crucial role in the subsequent
sections as it allows us to describe the limiting measure of the
rescaled fluctuations
by a stochastic process whose associated Malliavin
covariance matrix is non-degenerate.
\begin{lemma}\label{spanprop}
  Let $\tilde\bracol_k(0)=
    \left\{[\tilde X_{i_1},[\tilde X_{i_2},\dots,
      [\tilde X_{i_{k-1}},\tilde X_{i_k}]\dots]](0)
      \colon 1\leq i_1,\dots,i_k\leq m\right\}$.
  Then
  \begin{equation}\label{spaneq}
    \operatorname{span}\bigcup_{k=1}^n\tilde\bracol_k(0)=
    \operatorname{span}\left\{\frac{\pt}{\pt y^1}(0),
      \dots,\frac{\pt}{\pt y^{d_n}}(0)\right\}\;.
  \end{equation}
\end{lemma}
\begin{proof}
  We prove this lemma by induction. For the base case, we note that
  $\ord(\theta_*X_i-\tilde X_i)\leq 0$ implies $\tilde
  X_i(0)=(\theta_*X_i)(0)$. Hence, by property (\ref{cond1}) of an
  adapted chart $\theta$, we obtain
  \begin{equation*}
    \operatorname{span}\tilde\bracol_1(0)=
    \operatorname{span}\left\{\tilde X_1(0),\dots,\tilde
      X_m(0)\right\}=
    (\theta_*C_1)(0)=\operatorname{span}\left\{\frac{\pt}{\pt y^1}(0),
      \dots,\frac{\pt}{\pt y^{d_1}}(0)\right\}\;,
  \end{equation*}
  which proves (\ref{spaneq}) for $n=1$. Let us now assume the result
  for $n-1$. Due to
  $\ord(\theta_*X_i-\tilde X_i)\leq 0$ and
  using (\ref{eq:ordineq}) as well
  as the bilinearity of the Lie bracket, it follows that
  \begin{equation*}
    \ord\left(\theta_*[X_{i_1},[X_{i_2},\dots, [X_{i_{n-1}},X_{i_n}]\dots]]-
    [\tilde X_{i_1},[\tilde X_{i_2},\dots,
    [\tilde X_{i_{n-1}},\tilde X_{i_n}]\dots]]\right)\leq n-1\;.
  \end{equation*}
  Applying the induction hypothesis, we deduce that
  \begin{align*}
    &\left(\theta_*[X_{i_1},[X_{i_2},\dots, [X_{i_{n-1}},X_{i_n}]\dots]]-
    [\tilde X_{i_1},[\tilde X_{i_2},\dots,
    [\tilde X_{i_{n-1}},\tilde X_{i_n}]\dots]]\right)(0)\\
    &\qquad\qquad\in
    \operatorname{span}\left\{\frac{\pt}{\pt y^1}(0),
      \dots,\frac{\pt}{\pt y^{d_{n-1}}}(0)\right\}=
    \operatorname{span}\bigcup_{k=1}^{n-1}\tilde\bracol_k(0)\;.
  \end{align*}
  This gives
  \begin{equation*}
    \operatorname{span}\left\{\frac{\pt}{\pt y^1}(0),
      \dots,\frac{\pt}{\pt y^{d_{n}}}(0)\right\}=
    (\theta_*C_n)(0)\subset
    \operatorname{span}\bigcup_{k=1}^n\tilde\bracol_k(0)
  \end{equation*}
  and since $\ord\left([\tilde X_{i_1},[\tilde X_{i_2},\dots,
    [\tilde X_{i_{n-1}},\tilde X_{i_n}]\dots]]\right)\leq n$, the other
  inclusion holds as well. Thus, we have established
  equality, which concludes the induction step.
\end{proof}
The lemma allows us to prove the following proposition.
\begin{propn}\label{prop:sHlimit}
  The nilpotent approximations
  $\tilde X_1,\dots,\tilde X_m$ satisfy the strong
  H\"ormander condition everywhere on $\R^d$.
\end{propn}
\begin{proof}
  By definition, we have $d_N=d$, and Lemma~\ref{spanprop} implies that
  \begin{equation*}
    \operatorname{span}\bigcup_{k=1}^N\tilde\bracol_k(0)=
    \operatorname{span}\left\{\frac{\pt}{\pt y^1}(0),
      \dots,\frac{\pt}{\pt y^{d}}(0)\right\}=\R^d\;,
  \end{equation*}
  i.e. $\tilde X_1,\dots,\tilde X_m$ satisfy the strong H\"ormander
  condition at $0$. In particular, there are vector fields
  \begin{equation*}
    Y_1,\dots,Y_d\in \bigcup_{k=1}^N
    \left\{[\tilde X_{i_1},[\tilde X_{i_2},\dots,
      [\tilde X_{i_{k-1}},\tilde X_{i_k}]\dots]]
      \colon 1\leq i_1,\dots,i_k\leq m\right\}
  \end{equation*}
  such that $Y_1(0),\dots,Y_d(0)$ are linearly independent,
  i.e. $\operatorname{det}(Y_1(0),\dots,Y_d(0))\not= 0$.
  By continuity of the map
  $y\mapsto\operatorname{det}(Y_1(y),\dots,Y_d(y))$, it
  follows that there exists a neighbourhood $V_0$ of $0$ on which
  the vector fields $\tilde X_1,\dots,\tilde X_m$ satisfy the
  strong H\"ormander condition. Since the
  Lie bracket commutes with 
  pushforward, the homogeneity property
  $\left(\delta_\eps^{-1}\right)_* \tilde X_i =\eps^{-1/2}\tilde X_i$
  of the nilpotent approximations shows that the strong H\"ormander
  condition is in fact satisfied on all of $\R^d$.
\end{proof}

We conclude with an example.
\begin{eg0}\label{example}\rm
  Let $M=\R^2$ and fix $x=0$. Let $X_1$ and $X_2$ be the vector fields on
  $\R^2$ defined by
  \begin{equation*}
    X_1=\frac{\pt}{\pt x^1}+x^1\frac{\pt}{\pt x^2}
    \qquad\mbox{and}\qquad X_2=x^1\frac{\pt}{\pt x^1}\;,
  \end{equation*}
  with respect to Cartesian coordinates $(x^1,x^2)$ on $\R^2$.
  We compute
  \begin{equation*}
    [X_1,X_2]=\frac{\pt}{\pt x^1}-x^1\frac{\pt}{\pt x^2}
    \qquad\mbox{and}\qquad
    [X_1,[X_1,X_2]]=-2\frac{\pt}{\pt x^2}\;.
  \end{equation*}
  It follows that
  \begin{equation*}
    C_1(0)=C_2(0)=\operatorname{span}\left\{\frac{\pt}{\pt
        x^1}(0)\right\}\;,\enspace C_3(0)=\R^2
   \quad\mbox{and}\quad
   d_1=d_2=1\;,\enspace d_3=2\;.
  \end{equation*}
  We note that the Cartesian coordinates are not adapted to
  the filtration induced by $(X_1,X_2)$ at $0$ because,
  for instance, $\left((X_1)^2\,x_2\right)(0)=1$.
  Following the constructive proof
  of \cite[Corollary~3.1]{bianchini}, we find a global adapted chart
  $\theta\colon\R^2\to\R^2$ at $0$ given by
  \begin{equation*}
    \theta^1=x^1\quad\mbox{and}\quad
    \theta^2=-\frac{1}{2}(x^1)^2+x^2\;.
  \end{equation*}
  The corresponding weights are $w_1=1$, $w_2=3$ and the
  associated anisotropic dilation is
  \begin{equation*}
    \delta_\eps(y^1,y^2)=\left(\eps^{1/2}y^1,\eps^{3/2}y^2\right)\;,
  \end{equation*}
  where $(y^1,y^2)$ are Cartesian coordinates on our
  new copy of $\R^2$.
  For the pushforward vector fields of $X_1$ and $X_2$ by
  $\theta$, we obtain
  \begin{equation*}
    \theta_*X_1=\frac{\pt}{\pt y^1}
    \qquad\mbox{and}\qquad
    \theta_*X_2=
    y^1\left(\frac{\pt}{\pt y^1}-y^1\frac{\pt}{\pt y^2}\right)\;.
  \end{equation*}
  From this we can read off that
  \begin{equation*}
    \tilde X_1=\frac{\pt}{\pt y^1}
    \qquad\mbox{and}\qquad
    \tilde X_2=-\left(y^1\right)^2\frac{\pt}{\pt y^2}
  \end{equation*}
  because $y^1\frac{\pt}{\pt y^1}$ is a vector field
  of weight $0$. We observe that
  $\tilde X_1$ and $\tilde X_2$ are indeed homogeneous vector fields
  of weight $1$
  on $\R^2$ which satisfy the strong H\"ormander condition
  everywhere.
\end{eg0}

\section{Uniform non-degeneracy of the rescaled
Malliavin covariance matrices}
\label{sec:UND}
We prove the uniform non-degeneracy of suitably rescaled Malliavin
covariance matrices under the global condition
\begin{equation*}
  M=\R^d\quad\mbox{and}\quad X_0,X_1,\dots,X_m\in C_b^\infty(\R^d,\R^d)\;,
\end{equation*}
and the additional assumption that
$X_0(y)\in\operatorname{span}\{X_1(y),\dots,X_m(y)\}$ for all $y\in
\R^d$. We further assume that $\theta\colon\R^d\to\R^d$ is a global
diffeomorphism with bounded derivatives of all positive
orders and an adapted chart to
the filtration induced by the sub-Riemannian structure
$(X_1,\dots,X_m)$ at $x\in\R^d$ fixed.
Such a diffeomorphism always exists as \cite[Corollary~3.1]{bianchini}
guarantees the existence of an adapted chart $\tilde\theta\colon
U\to\R^d$ and due to \cite[Lemma~5.2]{palais}, we can construct a
global diffeomorphism $\theta\colon\R^d\to\R^d$ with bounded derivatives
of all positive
orders which agrees with $\tilde\theta$ on a small enough
neighbourhood of $x$ in $U$. We note that
$\theta_*X_0,\theta_*X_1,\dots,\theta_*X_m$ are also smooth bounded vector
fields on $\R^d$ with bounded derivatives of all orders.
In particular, to simplify the notation in
the subsequent analysis, we may assume $x=0$ and that
$\theta$ is the identity map. By Section~\ref{sec:graded}, this means
that, for Cartesian coordinates $(y_1,\dots,y_d)$ on $\R^d$ and
for all $n\in\{1,\dots,N\}$, we have
\begin{enumerate}[{\rm (i)}]
\item $\displaystyle C_n(0)=
  \operatorname{span}\left\{\frac{\pt}{\pt y^1}(0), 
    \dots,\frac{\pt}{\pt y^{d_n}}(0)\right\}\;,$ and
\item \label{cond2sec3}
  $\left(\operatorname{D} y^k\right)(x)=0$ for every
  differential operator
  $\operatorname{D}\in\{Y_1\cdots Y_j\colon Y_l\in C_{i_l}\mbox{ and }
  i_1+\dots+i_j\leq n\}$ and all $k>d_n\;.$
\end{enumerate}
Write $\langle\cdot,\cdot\rangle$ for the standard inner product on $\R^d$
and, for $n\in\{0,1,\dots,N\}$, denote the orthogonal complement of
$C_n(0)$ with respect to this inner product by $C_n(0)^\perp$.
As defined in the previous section, we further let
$\delta_\eps\colon\R^d\to\R^d$ be the anisotropic dilation induced by
the filtration
at $0$ and we consider the nilpotent approximations $\tilde
X_1,\dots,\tilde X_m$ of the vector fields $X_1,\dots,X_m$.

Let $(B_t)_{t\in[0,1]}$ be a Brownian motion in $\R^m$, which is
assumed to be realised as the coordinate process on the path space
$\{w\in C([0,1],\R^m)\colon w_0=0\}$ under Wiener measure $\pr$.
Define
$\underline{X}_0$ to be
the vector field on $\R^d$ given by
\begin{equation*}
  \underline X_0=X_0+\frac{1}{2}\sum_{i=1}^m\nabla_{X_i}X_i\;,
\end{equation*}
where $\nabla$ is the Levi-Civita connection with respect to the
Euclidean metric.
Under our global assumption,
the It\^o stochastic differential equation in $\R^d$
\begin{equation*}
  \db x_t^\eps=\sum_{i=1}^m \sqrt{\eps} X_i(x_t^\eps)\dd B_t^i+
  \eps\underline X_0(x_t^\eps)\dd t\;,\quad x_0^\eps=0
\end{equation*}
has a unique strong solution $(x_t^\eps)_{t\in[0,1]}$. Its law on
$\Omega=C([0,1],\R^d)$ is $\mu_\eps^0$.
We consider the rescaled diffusion process $(\tilde
x_t^\eps)_{t\in[0,1]}$ which is defined by
$\tilde x_t^\eps=\scale_\eps^{-1}(x_t^\eps)$.
It is the unique strong
solution of the It\^o stochastic differential equation
\begin{equation*}
  \db\tilde x_t^\eps=\sum_{i=1}^m\sqrt{\eps}
  \left(\left(\delta_\eps^{-1}\right)_*X_i\right)(\tilde x_t^\eps)\dd B_t^i
  +\eps\left(\left(\delta_\eps^{-1}\right)_*\underline X_0\right)
  (\tilde x_t^\eps)\dd t\;,\quad \tilde x_0^\eps=0\;.
\end{equation*}
Let us further look at
\begin{equation*}
  \db\tilde x_t=\sum_{i=1}^m \tilde X_i(\tilde x_t)\dd B_t^i+
  \underline{\tilde X}_0(\tilde x_t)\dd t\;,\quad \tilde x_0=0\;,
\end{equation*}
where $\underline{\tilde X}_0$ is the vector field on $\R^d$ defined
by
\begin{equation*}
  \underline{\tilde X}_0=
  \frac{1}{2}\sum_{i=1}^m\nabla_{\tilde X_i}\tilde X_i\;.
\end{equation*}
Due to the nice cascade structure discussed in Remark~\ref{cascade} and by
\cite[Proposition~1.3]{jamesSMC}, there exists a unique strong
solution $(\tilde x_t)_{t\in[0,1]}$ to this It\^o
stochastic differential equation in $\R^d$. We recall that
$\sqrt{\eps}\left(\delta_\eps^{-1}\right)_* X_i\to\tilde X_i$
as $\eps\to 0$ for all $i\in\{1,\dots,m\}$ and because
$X_0(y)\in\operatorname{span}
\{X_1(y),\dots,X_m(y)\}$ for all $y\in\R^d$, we further
have $\eps\left(\delta_\eps^{-1}\right)_* X_0\to 0$
as $\eps\to 0$. It follows that, for all $t\in[0,1]$,
\begin{equation}\label{convergence}
  \tilde x_t^\eps\to\tilde x_t\mbox{ as }\eps\to 0
  \mbox{ almost surely and in }L^p
  \mbox{ for all }p<\infty\;.
\end{equation}
For the Malliavin covariance matrices
$\tilde c_1^\eps$ of $\tilde x_1^\eps$ and $\tilde c_1$ of 
$\tilde x_1$, we also obtain that
\begin{equation}\label{convofMCM}
  \tilde c_1^\eps\to\tilde c_1\mbox{ as }\eps\to 0
  \mbox{ almost surely and in }L^p
  \mbox{ for all }p<\infty\;.
\end{equation}
Proposition~\ref{prop:sHlimit} shows that the nilpotent
approximations $\tilde X_1,\dots,\tilde X_m$ satisfy the strong
H\"ormander condition everywhere, which implies the
following non-degeneracy result.
\begin{cor}
  The Malliavin covariance matrix $\tilde c_1$ is
  non-degenerate, i.e. for all $p<\infty$, we have
  \begin{equation*}
    \E\left[\left|\det \left(\tilde c_1\right)^{-1} \right|^p
    \right]<\infty\;.
  \end{equation*}
\end{cor}
Hence, the rescaled diffusion processes $(\tilde x_t^\eps)_{t\in[0,1]}$
have a non-degenerate limiting diffusion process as $\eps\to 0$.
This observation is important in establishing the uniform
non-degeneracy of the rescaled Malliavin covariance matrices
$\tilde c_1^\eps$.
In the following, we first gain control over the leading-order terms of
$\tilde c_1^\eps$ as $\eps\to 0$, which
then allows us to show that the
minimal eigenvalue of $\tilde c_1^\eps$ can be uniformly bounded below on
a set of high probability. Using this property, we prove
Theorem~\ref{thm:UND} at the end of the section.

\subsection{Properties of the rescaled Malliavin
covariance matrix}\label{section3dot1}
Let $(\tilde v_t^\eps)_{t\in[0,1]}$ be the unique stochastic process
in $\R^d\otimes (\R^d)^*$ such that
$(\tilde x_t^\eps,\tilde v_t^\eps)_{t\in[0,1]}$ is the strong solution
of the following system of It\^o stochastic differential equations
starting from $(\tilde x_0^\eps,\tilde v_0^\eps)=(0,I)$.
\begin{align*}
  \db\tilde x_t^\eps&=\sum_{i=1}^m\sqrt{\eps}
  \left(\left(\delta_\eps^{-1}\right)_*X_i\right)(\tilde x_t^\eps)\dd B_t^i
  +\eps\left(\left(\delta_\eps^{-1}\right)_*\underline X_0\right)
  (\tilde x_t^\eps)\dd t\\
  \db\tilde v_t^\eps&=-\sum_{i=1}^m\sqrt{\eps}\tilde v_t^\eps\nabla
  \left(\left(\delta_\eps^{-1}\right)_*X_i\right)(\tilde x_t^\eps)\dd B_t^i
  -\eps\tilde v_t^\eps\left(\nabla
    \left(\left(\delta_\eps^{-1}\right)_*\underline X_0\right)-
    \sum_{i=1}^m\left(\nabla \left(
        \left(\delta_\eps^{-1}\right)_*X_i\right)\right)^2\right)
  (\tilde x_t^\eps)\dd t
\end{align*}
The Malliavin covariance matrix $\tilde c_t^\eps$ of the rescaled
random variable $\tilde x_t^\eps$ can then be expressed as
\begin{equation*}
  \tilde c_t^\eps=\sum_{i=1}^m\int_0^t
  \left(\tilde
    v_s^\eps\left(\sqrt{\eps}\left(\delta_\eps^{-1}\right)_*
      X_i\right)(\tilde x_s^\eps)\right)\otimes
  \left(\tilde
    v_s^\eps\left(\sqrt{\eps}\left(\delta_\eps^{-1}\right)_*
      X_i\right)(\tilde x_s^\eps)\right)\dd s\;.
\end{equation*}
It turns out that we obtain a more tractable expression for
$\tilde c_t^\eps$ if we write
it in terms of $(x_t^\eps,v_t^\eps)_{t\in[0,1]}$, which is the unique strong
solution of the following system of It\^o stochastic differential
equations.
\begin{align*}
  \db x_t^\eps&=\sum_{i=1}^m\sqrt{\eps} X_i(x_t^\eps)\dd B_t^i
  +\eps\underline X_0(x_t^\eps)\dd t\;,\quad x_0^\eps=0\\
  \db v_t^\eps&=-\sum_{i=1}^m\sqrt{\eps}v_t^\eps\nabla X_i(x_t^\eps)\dd B_t^i
  -\eps v_t^\eps\left(\nabla\underline X_0-
    \sum_{i=1}^m(\nabla X_i)^2\right)(x_t^\eps)\dd t\;,
  \quad v_0^\eps=I
\end{align*}
One can check that the stochastic processes
$(v_t^\eps)_{t\in[0,1]}$ and $(\tilde v_t^\eps)_{t\in[0,1]}$
are related by $\tilde v_t^\eps=\scalemat_\eps^{-1}
v_t^\eps\scalemat_\eps$\;, where the map $\scale_\eps$ is understood as an
element in $\R^d\otimes (\R^d)^*$. This implies that
\begin{equation}\label{eq:rescaledMCM}
  \tilde c_t^\eps=\sum_{i=1}^m\int_0^t\left(
  \sqrt{\eps}\scale_\eps^{-1}\left(v_s^\eps X_i(x_s^\eps)\right)\right)
  \otimes\left(
  \sqrt{\eps}\scale_\eps^{-1}\left(v_s^\eps X_i(x_s^\eps)\right)\right)\dd s\;.
\end{equation}
We are interested in gaining control over
the leading-order terms of $\tilde
c_1^\eps$ as $\eps\to 0$. In the corresponding
analysis, we frequently use the
lemma stated below.
\begin{lemma}\label{lem:vXid}
  Let $Y$ be a smooth vector field on $\R^d$. Then
  \begin{equation*}
    \db(v_t^\eps Y(x_t^\eps))=\sum_{i=1}^m\sqrt{\eps} v_t^\eps
    [X_i,Y](x_t^\eps)\dd B_t^i +\eps v_t^\eps\left(
      [X_0,Y]+\frac{1}{2}\sum_{i=1}^m
      \left[X_i,\left[X_i,Y\right]\right]\right)(x_t^\eps)\dd t\;.
  \end{equation*}
\end{lemma}
\begin{proof}
  To prove this identity, we switch to the Stratonovich setting.
  The system of Stratonovich stochastic differential equations
  satisfied by the processes $(x_t^\eps)_{t\in[0,1]}$ and
  $(v_t^\eps)_{t\in[0,1]}$ is
  \begin{align*}
    \stb x_t^\eps &= \sum_{i=1}^m\sqrt{\eps} X_i(x_t^\eps)\std B_t^i + \eps
    X_0(x_t^\eps)\dd t\;,\quad x_0^\eps=0\\
    \stb v_t^\eps &= -\sum_{i=1}^m\sqrt{\eps}v_t^\eps \nabla
    X_i(x_t^\eps)\std B_t^i - \eps
    v_t^\eps\nabla X_0(x_t^\eps)\dd t\;,\quad v_0^\eps=I.
  \end{align*}
  By the product rule, we have
  \begin{equation*}
    \stb(v_t^\eps Y(x_t^\eps))
    =(\stb v_t^\eps)Y(x_t^\eps)+
     v_t^\eps\nabla Y(x_t^\eps)\std x_t^\eps\;.
  \end{equation*}
  Using
  \begin{equation*}
    (\stb v_t^\eps)Y(x_t^\eps)=
      -\sum_{i=1}^m\sqrt{\eps}v_t^\eps
       \nabla X_i(x_t^\eps)Y(x_t^\eps)\std B_t^i
      -\eps v_t^\eps\nabla X_0(x_t^\eps)Y(x_t^\eps)\dd t 
  \end{equation*}
  as well as
  \begin{equation*}
    v_t^\eps\nabla Y(x_t^\eps)\std x_t^\eps=
      \sum_{i=1}^m\sqrt{\eps} v_t^\eps\nabla
      Y(x_t^\eps)X_i(x_t^\eps)\std B_t^i +
      \eps v_t^\eps\nabla Y(x_t^\eps)X_0(x_t^\eps)\dd t
  \end{equation*}
  yields the identity
  \begin{equation*}
    \stb(v_t^\eps Y(x_t^\eps))=
    \sum_{i=1}^m\sqrt{\eps} v_t^\eps [X_i,Y](x_t^\eps)\std B_t^i+
    \eps v_t^\eps[X_0,Y](x_t^\eps)\dd t\;.
  \end{equation*}
  It remains to change back to the It\^o setting. We compute that,
  for $i\in\{1,\dots,m\}$,
  \begin{align*}
    &\db\left[\sqrt{\eps} v^\eps
      [X_i,Y](x^\eps),B^i\right]_t\\
    &\quad=\sum_{j=1}^m\eps v_t^\eps \nabla [X_i,Y](x_t^\eps)
      X_j(x_t^\eps)\dd[B^j,B^i]_t
      -\sum_{j=1}^m\eps v_t^\eps \nabla X_j(x_t^\eps)
      [X_i,Y](x_t^\eps)\dd[B^j,B^i]_t\\
    &\quad=\eps v_t^\eps \nabla [X_i,Y](x_t^\eps)
      X_i(x_t^\eps)\dd t
      -\eps v_t^\eps \nabla X_i(x_t^\eps)
      [X_i,Y](x_t^\eps)\dd t\\
    &\quad= \eps v_t^\eps [X_i,[X_i,Y]](x_t^\eps)\dd t
  \end{align*}
  and the claimed result follows.
\end{proof}
The next lemma, which is enough for our purposes,
does not provide an explicit expression for the
leading-order terms of $\tilde c_1^\eps$. However, its proof shows how
one could recursively obtain these expressions if one wishes to do so. 
To simplify notations, we introduce $(B_t^0)_{t\in[0,1]}$ with $B_t^0=t$.
\begin{lemma}\label{lem:leadterm}
  For every $n\in\{1,\dots,N\}$,
  there are finite collections of vector fields
  \begin{align*}
    \mathcal{B}_n=
    \left\{Y_{j_1,\dots,j_k}^{(n,i)}\colon 1\leq k\leq n,
      0\leq j_1,\dots,j_k\leq m,1\leq i\leq m\right\}&\subset C_{n+1}
    \quad\mbox{and}\\
    \tilde{\mathcal{B}}_n=
    \left\{\tilde Y_{j_1,\dots,j_k}^{(n,i)}\colon 1\leq k\leq n,
      0\leq j_1,\dots,j_k\leq m,1\leq i\leq m\right\}&\subset C_{n+2}
  \end{align*}
  such that, for all $u\in C_n(0)^\perp$ and all $i\in\{1,\dots,m\}$,
  we have that, for all $\eps>0$,
  \begin{align*}
    &\left\langle u,\eps^{-n/2} v_t^\eps X_i(x_t^\eps)\right\rangle\\
    &\quad=\left\langle u,
    \sum_{k=1}^n\sum_{j_1,\dots,j_k=0}^m
    \int_0^t\int_0^{t_2}\dots\int_0^{t_k}v_s^\eps
    \left(Y_{j_1,\dots,j_k}^{(n,i)}+
      \sqrt{\eps}\,\tilde Y_{j_1,\dots,j_k}^{(n,i)}\right)(x_s^\eps)
    \dd B_s^{j_k}\dd B_{t_k}^{j_{k-1}}\dots\dd B_{t_2}^{j_1}\right\rangle\;.
  \end{align*}
\end{lemma}
\begin{proof}
  We prove this result iteratively over $n$. For all $u\in
  C_1(0)^\perp$, we have $\langle u, X_i(0)\rangle=0$
  because $C_1(0)=\operatorname{span}\{X_1(0),\dots,X_m(0)\}$.
  From Lemma~\ref{lem:vXid}, it then follows that
  \begin{align*}
    &\left\langle u,\eps^{-1/2}v_t^\eps X_i(x_t^\eps)\right\rangle\\
    &\quad=\left\langle u,\sum_{j=1}^m\int_0^t
    v_s^\eps[X_j,X_i](x_s^\eps)\dd B_s^j+\int_0^t\sqrt{\eps} v_s^\eps\left(
      [X_0,X_i]+\frac{1}{2}\sum_{j=1}^m[X_j,[X_j,X_i]]\right)(x_s^\eps)\dd
    s \right\rangle\;.
  \end{align*}
  This gives us the claimed result for $n=1$ with
  \begin{align*}
    Y_j^{(1,i)}&=
    \begin{cases}
      0 & \mbox{if } j=0 \\
      [X_j,X_i]& \mbox{if } 1\leq j\leq m
    \end{cases}
    \qquad\mbox{and}\\
    \tilde Y_j^{(1,i)}&=
    \begin{cases}
      [X_0,X_i]+\frac{1}{2}\sum_{l=1}^m[X_l,[X_l,X_i]] & \mbox{if } j=0 \\
      0 & \mbox{otherwise}
    \end{cases}\;.
  \end{align*}
  Let us now assume the result to be true for $n-1$.
  Due to $C_{n}(0)^\perp\subset C_{n-1}(0)^\perp$,
  the corresponding identity also holds for all
  $u\in C_{n}(0)^\perp$. Using
  Lemma~\ref{lem:vXid}, we obtain that
  \begin{align*}
    v_s^\eps Y_{j_1,\dots,j_k}^{(n-1,i)}(x_s^\eps)=
    Y_{j_1,\dots,j_k}^{(n-1,i)}(0)&+
    \sum_{j=1}^m\int_0^s\sqrt{\eps}
    v_r^\eps\left[X_j,Y_{j_1,\dots,j_k}^{(n-1,i)}\right](x_r^\eps)\dd B_r^j\\
    &+\int_0^s\eps v_r^\eps\left(
      \left[X_0,Y_{j_1,\dots,j_k}^{(n-1,i)}\right]+
      \frac{1}{2}\sum_{j=1}^m\left[X_j,
        \left[X_j,Y_{j_1,\dots,j_k}^{(n-1,i)}\right]\right]\right)
    (x_r^\eps)\dd r\;.
  \end{align*}
  Note that $Y_{j_1,\dots,j_k}^{(n-1,i)}\in C_{n}$
  implies
  $\langle u, Y_{j_1,\dots,j_k}^{(n-1,i)}(0)\rangle=0$ for
  all $u\in C_n(0)^\perp$. We further observe that
  \begin{align*}
    \left[X_j,Y_{j_1,\dots,j_k}^{(n-1,i)}\right],
    \tilde Y_{j_1,\dots,j_k}^{(n-1,i)}&\in C_{n+1}\quad\mbox{as well as}\\
    \left[X_0,Y_{j_1,\dots,j_k}^{(n-1,i)}\right]+
      \frac{1}{2}\sum_{j=1}^m\left[X_j,
        \left[X_j,Y_{j_1,\dots,j_k}^{(n-1,i)}\right]\right]&\in C_{n+2}
  \end{align*}
  and collecting terms shows that the claimed
  result is also true for $n$.
\end{proof}
These expressions allow us to characterise the rescaled Malliavin
covariance matrix $\tilde c_1^\eps$ because,
for all $n\in\{0,1,\dots,N-1\}$ and all
$u\in C_{n+1}(0)\cap C_{n}(0)^\perp$, we have
\begin{equation}\label{eq:exp4MCM}
  \left\langle u,\tilde c_1^\eps u\right\rangle=
  \sum_{i=1}^m\int_0^1\left\langle u,\eps^{-n/2}
    v_t^\eps X_i(x_t^\eps)\right\rangle^2\dd t\;.
\end{equation}
By the convergence result (\ref{convofMCM}), it follows that, for
$u\in C_{1}(0)$,
\begin{equation*}
  \left\langle u,\tilde c_1 u\right\rangle=
  \lim_{\eps\to 0}\left\langle u,\tilde c_1^\eps u\right\rangle=
  \sum_{i=1}^m\int_0^1\left\langle u, X_i(0)\right\rangle^2\dd t
\end{equation*}
and from Lemma~\ref{lem:leadterm}, we deduce that, for
all $n\in\{1,\dots,N-1\}$ and all $u\in C_{n+1}(0)\cap C_{n}(0)^\perp$,
\begin{equation}\label{doublestar}
  \left\langle u,\tilde c_1 u\right\rangle=
  \sum_{i=1}^m\int_0^1\left\langle u,
    \sum_{k=1}^n\sum_{j_1,\dots,j_k=0}^m
    \int_0^t\int_0^{t_2}\dots\int_0^{t_k}
    Y_{j_1,\dots,j_k}^{(n,i)}(0)
    \dd B_s^{j_k}\dd B_{t_k}^{j_{k-1}}\dots\dd
    B_{t_2}^{j_1}\right\rangle^2\dd t\;,
\end{equation}
which describes the limiting Malliavin covariance matrix $\tilde c_1$
uniquely.

\subsection{Uniform non-degeneracy of the rescaled
  Malliavin covariance matrices}
By definition, the Malliavin covariance matrices $\tilde c_1^\eps$ and
$\tilde c_1$ are symmetric tensors. Therefore, their matrix representations are
symmetric in any basis and we can think of them as symmetric matrices.
Let $\mineig^\eps$ and $\mineig$ denote the minimal eigenvalues of
$\tilde c_1^\eps$ and $\tilde c_1$, respectively.
As we frequently use
the integrals from
Lemma~\ref{lem:leadterm}, it is convenient to consider the
stochastic processes
$(I_t^{(n,i),+})_{t\in[0,1]}$,
$(I_t^{(n,i),-})_{t\in[0,1]}$ and
$(\tilde I_t^{(n,i)})_{t\in[0,1]}$ given by
\begin{align*}
  I_t^{(n,i),+}&=
  \sum_{k=1}^n\sum_{j_1,\dots,j_k=0}^m
  \int_0^t\int_0^{t_2}\dots\int_0^{t_k}\left(v_s^\eps
  Y_{j_1,\dots,j_k}^{(n,i)}(x_s^\eps)+Y_{j_1,\dots,j_k}^{(n,i)}(0)\right)
  \dd B_s^{j_k}\dd B_{t_k}^{j_{k-1}}\dots\dd B_{t_2}^{j_1}\;,\\
  I_t^{(n,i),-}&=
  \sum_{k=1}^n\sum_{j_1,\dots,j_k=0}^m
  \int_0^t\int_0^{t_2}\dots\int_0^{t_k}\left(v_s^\eps
  Y_{j_1,\dots,j_k}^{(n,i)}(x_s^\eps)-Y_{j_1,\dots,j_k}^{(n,i)}(0)\right)
  \dd B_s^{j_k}\dd B_{t_k}^{j_{k-1}}\dots\dd B_{t_2}^{j_1}\;,
  \enspace\mbox{and}\\
  \tilde I_t^{(n,i)}&=
  \sum_{k=1}^n\sum_{j_1,\dots,j_k=0}^m
  \int_0^t\int_0^{t_2}\dots\int_0^{t_k}v_s^\eps
  \tilde Y_{j_1,\dots,j_k}^{(n,i)}(x_s^\eps)
  \dd B_s^{j_k}\dd B_{t_k}^{j_{k-1}}\dots\dd B_{t_2}^{j_1}\;.
\end{align*}
For $\alpha,\beta,\gamma,\delta>0$, define subspaces of the path
space $\{w\in C([0,1],\R^m)\colon w_0=0\}$ by
\begin{align*}
  \Omega^1(\alpha)&=\{\mineig\geq 2\alpha\}\;,\\
  \Omega_\eps^2(\beta,\gamma)&=
  \left\{\sup_{0\leq t\leq
      1}\left|I_t^{(n,i),+}\right|\leq\beta^{-1}\;,\enspace
    \sup_{0\leq t\leq 1}\left|\tilde I_t^{(n,i)}\right|\leq\gamma^{-1}
    \colon 1\leq i\leq m, 1\leq n\leq N\right\}\;,
  \enspace\mbox{and}\\
  \Omega_\eps^3(\delta)&=
  \left\{\sup_{0\leq t\leq 1}|x_t^\eps|\leq\delta\;,\enspace
    \sup_{0\leq t\leq 1}|v_t^\eps-I|\leq\delta\right\}\cup
    \left\{\sup_{0\leq t\leq 1}\left|I_t^{(n,i),-}\right|\leq\delta
    \colon 1\leq i\leq m, 1\leq n\leq N\right\}\;.
\end{align*}
Note that the events $\Omega_\eps^2(\beta,\gamma)$ and
$\Omega_\eps^3(\delta)$ depend on $\eps$ as the processes
$(I_t^{(n,i),+})_{t\in[0,1]}$,
$(I_t^{(n,i),-})_{t\in[0,1]}$ and
$(\tilde I_t^{(n,i)})_{t\in[0,1]}$ depend on $\eps$.
We show that, for suitable choices of
$\alpha,\beta,\gamma$ and $\delta$,
the rescaled Malliavin covariance matrices $\tilde c_1^\eps$
behave nicely on the set
\begin{equation*}
  \Omega(\alpha,\beta,\gamma,\delta,\eps)=
  \Omega^1(\alpha)\cap \Omega_\eps^2(\beta,\gamma)\cap \Omega_\eps^3(\delta)
\end{equation*}
and that
its complement is a set of small probability in the limit $\eps\to 0$.
As we are only interested in small values of
$\alpha,\beta,\gamma,\delta$ and $\eps$,
we may make the non-restrictive assumption that
$\alpha,\beta,\gamma,\delta,\eps<1$.
\begin{lemma}\label{lem:eigest}
  There exist positive constants $\chi$ and $\kappa$,
  which do not depend on $\eps$, such that if
  \begin{equation*}
    \chi\eps^{1/6}\leq\alpha\;,
    \qquad\beta=\gamma=\alpha
    \qquad\mbox{and}\qquad\delta=\kappa\alpha^2
  \end{equation*}
  then, on $\Omega(\alpha,\beta,\gamma,\delta,\eps)$,
  it holds true that
  \begin{equation*}
    \mineig^\eps\geq\frac{1}{2}\mineig\;.
  \end{equation*}
\end{lemma}
\begin{proof}
  Throughout, we shall assume that we are on the event
  $\Omega(\alpha,\beta,\gamma,\delta,\eps)$. Let
  \begin{equation*}
    R^\eps(u)=\frac{\left\langle u,\tilde c_1^\eps u\right\rangle}
      {\langle u,u\rangle}
    \qquad\mbox{and}\qquad
    R(u)=\frac{\left\langle u,\tilde c_1 u\right\rangle}
      {\langle u,u\rangle}
  \end{equation*}
  be the Rayleigh-Ritz quotients of the rescaled Malliavin covariance
  matrix $\tilde c_1^\eps$
  and of the limiting Malliavin covariance matrix $\tilde c_1$,
  respectively. As a consequence of the Min-Max Theorem, we have
  \begin{equation*}
    \mineig^\eps=\min\{R^\eps(u)\colon u\not=0\}
    \qquad\mbox{as well as}\qquad
    \mineig=\min\{R(u)\colon u\not=0\}\;.
  \end{equation*}
  Since $\lambda_{\rm min}\geq 2\alpha$, it suffices to establish that
  $|R^\eps(u)-R(u)|\leq \alpha$ for all $u\not= 0$. Set
  \begin{equation*}
    K = \max_{1\leq i\leq m}\sup_{y\in\R^d} |X_i(y)|\;,\qquad
    L = \max_{1\leq i\leq m}\sup_{y\in\R^d} |\nabla X_i(y)|
  \end{equation*}
  and note that the global condition ensures $K,L<\infty$.
  Using the Cauchy-Schwarz inequality, we deduce that, for
  $u\in C_1(0)\setminus\{0\}$,
  \begin{align*}
    |R^\eps(u)-R(u)|&\leq
    \frac{\displaystyle\sum_{i=1}^m\int_0^1\left|
        \left\langle u,v_t^\eps X_i(x_t^\eps)\right\rangle^2-
        \left\langle u,X_i(0)\right\rangle^2
      \right|\dd t}{\langle u,u \rangle}\\
    &\leq \sum_{i=1}^m\int_0^1| v_t^\eps
     X_i(x_t^\eps)+X_i(0)| | v_t^\eps X_i(x_t^\eps) - X_i(0)|\dd t\\
    &\leq m((1+\delta)K+K)(\delta K+\delta L)\;.
  \end{align*}
  Applying Lemma~\ref{lem:leadterm}
  as well as the expressions
  (\ref{eq:exp4MCM}) and (\ref{doublestar}),
  we obtain in a similar way
  that, for all $n\in\{1,\dots,N-1\}$ and all
  non-zero $u\in C_{n+1}(0)\cap C_n(0)^\perp$,
  \begin{align*}
    |R^\eps(u)-R(u)|&\leq
    \sum_{i=1}^m\int_0^1
    \left|I_t^{(n,i),+}+\sqrt{\eps}\tilde I_t^{(n,i)}\right|
    \left|I_t^{(n,i),-}+\sqrt{\eps}\tilde I_t^{(n,i)}\right|\dd t\\
    &\leq m\left(\beta^{-1}+\sqrt{\eps}\gamma^{-1}\right)
     \left(\delta +\sqrt{\eps}\gamma^{-1}\right)\;.
  \end{align*}
  It remains to consider the cross-terms.
  For $n_1,n_2\in\{1,\dots,N-1\}$ and
  $u^1\in C_{n_1+1}(0)\cap C_{n_1}(0)^\perp$ as well as
  $u^2\in C_{n_2+1}(0)\cap C_{n_2}(0)^\perp$, we polarise
  (\ref{eq:exp4MCM}) to conclude that
  \begin{align*}
    \frac{\left\langle u^1,\tilde c_1^\eps u^2\right\rangle-
    \left\langle u^1,\tilde c_1 u^2\right\rangle}{|u^1||u^2|}&\leq
  \sum_{i=1}^m\int_0^1
  \left|\frac{I_t^{(n_1,i),+}+I_t^{(n_1,i),-}}{2}+\sqrt{\eps}
    \tilde I_t^{(n_1,i)}\right|\left|I_t^{(n_2,i),-}+\sqrt{\eps}
    \tilde I_t^{(n_2,i)}\right|\dd t\\
  &\qquad+\sum_{i=1}^m\int_0^1
  \left|I_t^{(n_1,i),-}+\sqrt{\eps}
    \tilde I_t^{(n_1,i)}\right|
  \left|\frac{I_t^{(n_2,i),+}-I_t^{(n_2,i),-}}{2}\right|\dd t\\
  &\leq m\left(\beta^{-1}+\delta+\sqrt{\eps}\gamma^{-1}\right)
  \left(\delta+\sqrt{\eps}\gamma^{-1}\right)\;.
  \end{align*}
  Similarly, if $n_1=0$ and $n_2\in\{1,\dots,N-1\}$, we see that
  \begin{equation*}
    \frac{\left\langle u^1,\tilde c_1^\eps u^2\right\rangle-
    \left\langle u^1,\tilde c_1 u^2\right\rangle}{|u^1||u^2|}\leq
    m\left((1+\delta)K\left(\delta+\sqrt{\eps}\gamma^{-1}\right) + 
    (\delta K+\delta L)\left(\frac{\beta^{-1}+\delta}{2}\right)\right)\;.
  \end{equation*}
  Writing a general non-zero $u\in\R^d$ in its orthogonal sum
  decomposition and combining all the above
  estimates gives
  \begin{equation*}
    |R^\eps(u)-R(u)|\leq
    \kappa_1\delta + \kappa_2\beta^{-1}\delta+
    \kappa_3\sqrt{\eps}\beta^{-1}\gamma^{-1}+
    \kappa_4\eps\gamma^{-2}
  \end{equation*}
  for some constants $\kappa_1,\kappa_2,\kappa_3$ and $\kappa_4$,
  which depend on $K,L$ and $m$ but which are independent of
  $\alpha,\beta,\gamma,\delta$ and $\eps$.
  If we now choose $\kappa$ and $\chi$ in such a way
  that both
  $\kappa\leq 1/(4\max\{\kappa_1,\kappa_2\})$ and
  $\chi^3\geq 4\max\{\kappa_3,\kappa_4^{1/2}\}$,
  and provided that $\chi\eps^{1/6}\leq\alpha$,
  $\beta=\gamma=\alpha$ as well as $\delta=\kappa\alpha^2$,
  then
  \begin{equation*}
    \kappa_1\delta + \kappa_2\beta^{-1}\delta+
    \kappa_3\sqrt{\eps}\beta^{-1}\gamma^{-1}+
    \kappa_4\eps\gamma^{-2}\leq
    \kappa_1\kappa\alpha^2 + \kappa_2\kappa\alpha+
    \kappa_3\chi^{-3}\alpha + \kappa_4\chi^{-6}\alpha^4
    \leq\alpha\;.
  \end{equation*}
  Since $\kappa$ and $\chi$ can always be chosen to be positive,
  the desired result follows.
\end{proof}
As a consequence of this lemma, we are able to control
$\det \left(\tilde c_1^\eps\right)^{-1}$ on the good set
$\Omega(\alpha,\beta,\gamma,\delta,\eps)$. This allows us
to prove Theorem~\ref{thm:UND}.
\begin{proof}[Proof of Theorem~\ref{thm:UND}]
  We recall that by Proposition~\ref{prop:sHlimit}, the nilpotent
  approximations $\tilde X_1,\dots,\tilde X_m$ satisfy the strong
  H\"ormander condition everywhere on $\R^d$. The proof of
  \cite[Theorem~4.2]{jamesSMC}
  then shows that
  \begin{equation}\label{dagger}
    \mineig^{-1}\in L^p(\pr)\;,\quad\mbox{for all}\quad p<\infty\;.
  \end{equation}
  By the Markov inequality,
  this integrability result implies that, for
  all $p<\infty$, there exist constants $D(p)<\infty$ such that
  \begin{equation}\label{eq:est4alpha}
    \pr\left(\Omega^1(\alpha)^c\right)\leq D(p)\alpha^p.
  \end{equation}
  Using the Burkholder-Davis-Gundy inequality and Jensen's inequality,
  we further show that, for all $p<\infty$, there are constants
  $E_1(p),E_2(p)<\infty$ such that
  \begin{equation*}
    \E\left[\sup_{0\leq t\leq 1}|x_t^\eps|^p\right]\leq
    E_1(p)\eps^{p/2}\qquad\mbox{and}\qquad
    \E\left[\sup_{0\leq t\leq 1}|v_t^\eps-I|^p\right]\leq
    E_2(p)\eps^{p/2}\;.
  \end{equation*}
  Similarly, by repeatedly applying the Burkholder-Davis-Gundy
  inequality and Jensen's inequality, we also see that, for all
  $p<\infty$ and for all
  $n\in\{1,\dots,N\}$ and $i\in\{1,\dots,m\}$,
  there exist constants $E^{(n,i)}(p)<\infty$ and
  $D^{(n,i)}(p),\tilde D^{(n,i)}(p)<\infty$ such that
  \begin{equation*}
    \E\left[\sup_{0\leq t\leq 1}\left|I_t^{(n,i),-}\right|^p\right]\leq
    E^{(n,i)}(p)\eps^{p/2}
  \end{equation*}
  as well as
  \begin{equation*}
    \E\left[\sup_{0\leq t\leq 1}\left|I_t^{(n,i),+}\right|^p\right]\leq
    D^{(n,i)}(p)
    \qquad\mbox{and}\qquad
    \E\left[\sup_{0\leq t\leq 1}\left|\tilde I_t^{(n,i)}\right|^p\right]\leq
    \tilde D^{(n,i)}(p)\;.
  \end{equation*}
  As the sets $\Omega_\eps^2(\beta,\gamma)$ and
  $\Omega_\eps^3(\delta)$ are defined by only finitely many constraints,
  the bounds established above and the Markov inequality imply that,
  for all $p<\infty$, there are constants $D(p)<\infty$ and
  $E(p)<\infty$ such that
  \begin{align}
    \label{eq:est4beta}
    \pr\left(\Omega_\eps^2(\beta,\gamma)^c\right)&\leq
    D(p)\left(\beta^p+\gamma^p\right)\qquad\mbox{and}\\
    \label{eq:est4delta}
    \pr\left(\Omega_\eps^3(\delta)^c\right)&\leq E(p)\delta^{-p}\eps^{p/2}\;.
  \end{align}
  Moreover, from the Kusuoka-Stroock estimate, cf. \cite{AKS}, as stated
  by Watanabe \cite[Theorem~3.2]{watanabeUND}, we know that
  there exist a positive integer $S$ and,
  for all $p<\infty$, constants $C(p)<\infty$
  such that, for all $\eps\in(0,1]$,
  \begin{equation*}
    \|\det(\tilde c_1^\eps)^{-1}\|_p=
    \left(\E\left[\left|\det\left(\tilde c_1^\eps\right)^{-1}
      \right|^p\right]\right)^{1/p}\leq C(p)\eps^{-S/2}\;.
  \end{equation*}
  Let us now choose $\alpha=\chi^{3/4}\eps^{1/8}$,
  $\beta=\gamma=\alpha$ and $\delta=\kappa\alpha^2$.
  We note that $\chi\eps^{1/6}=\alpha^{4/3}\leq\alpha$ and
  hence, from Lemma~\ref{lem:eigest} it follows that
  \begin{equation*}
    \mineig^\eps\geq\frac{1}{2}\mineig
  \end{equation*}
  on $\Omega(\alpha,\beta,\gamma,\delta,\eps)$.
  Thus, we have
  \begin{equation*}
    \det(\tilde c_1^\eps)^{-1}
    \ind_{\Omega(\alpha,\beta,\gamma,\delta,\eps)}
    \leq (\mineig^\eps)^{-d}
    \ind_{\Omega(\alpha,\beta,\gamma,\delta,\eps)}
    \leq 2^d\mineig^{-d}
    \ind_{\Omega(\alpha,\beta,\gamma,\delta,\eps)}
  \end{equation*}
  and therefore,
  \begin{equation*}
    \det(\tilde c_1^\eps)^{-1}\leq 2^d\mineig^{-d}+\det(\tilde c_1^\eps)^{-1}
    \left(
      \ind_{\Omega^1(\alpha)^c}+\ind_{\Omega_\eps^2(\beta,\gamma)^c}+
      \ind_{\Omega_\eps^3(\delta)^c}
    \right)\;.
  \end{equation*}
  Using the H\"older
  inequality, the Kusuoka-Stroock estimate as well as the estimates
  (\ref{eq:est4alpha}), (\ref{eq:est4beta}) and (\ref{eq:est4delta}),
  we further deduce that, for all $q,r<\infty$,
  \begin{align*}
    \|\det(\tilde c_1^\eps)^{-1}\|_p&\leq 2^d\|\lambda_{\rm min}^{-1}\|_p^d+
    C(2p)\eps^{-S/2}\left(\pr\left(\Omega^1(\alpha)^c\right)^{1/2p}+
                   \pr\left(\Omega_\eps^2(\beta,\gamma)^c\right)^{1/2p}+
                         \pr\left(\Omega_\eps^3(\delta)^c\right)^{1/2p} \right)\\
    &\leq 2^d\|\lambda_{\rm min}^{-1}\|_p^d+C(2p)\eps^{-S/2}
    \left((D(q)\alpha^q)^{1/2p}+
    \left(E(r)\delta^{-r}\eps^{r/2}\right)^{1/2p}\right).
  \end{align*}
  Hence, we would like to choose $q$ and $r$ in such a way that we can
  control both $\eps^{-S/2}\alpha^{q/2p}$ and
  $\eps^{-S/2}\delta^{-r/2p}\eps^{r/4p}$. Since $\delta=
  \kappa\alpha^2$ and $\alpha=\chi^{3/4}\eps^{1/8}$, we have
  \begin{equation*}
    \eps^{-S/2}\alpha^{q/2p}=\chi^{3q/8p}\eps^{-S/2+q/16p}
    \quad\mbox{as well as}\quad
    \eps^{-S/2}\delta^{-r/2p}\eps^{r/4p}=
    \left(\kappa\chi^{3/2}\right)^{-r/2p}\eps^{-S/2+r/8p}\;.
  \end{equation*}
  Thus, picking $q=8pS$ and $r=4pS$ ensures both terms remain
  bounded as $\eps\to 0$ and we obtain
  \begin{equation*}
    \|\det(\tilde c_1^\eps)^{-1}\|_p
    \leq 2^d\|\lambda_{\rm min}^{-1}\|_p^d+C(2p)
      \left(D(8pS,\chi)^{1/2p}+E(4pS,\kappa,\chi)^{1/2p}\right)\;.
  \end{equation*}
  This together with the integrability (\ref{dagger})
  of $\lambda_{\rm min}^{-1}$ implies the
  uniform non-degeneracy of the rescaled Malliavin covariance matrices
  $\tilde c_1^\eps$.
\end{proof}

\section{Convergence of the diffusion bridge measures}
\label{sec:convofdiffmeas}
We prove Theorem~\ref{propn:convofbridges} in
this section with
the extension to Theorem~\ref{thm:convofbridges} left to Section
\ref{pp:local}. For our analysis, we adapt
the Fourier transform argument presented in \cite{BMN} to allow for
the higher-order scaling $\delta_\eps$. As in Section~\ref{sec:UND},
we may assume that the sub-Riemannian structure $(X_1,\dots,X_m)$ has
already been pushed forward by the global diffeomorphism
$\theta\colon\R^d\to\R^d$ which is an adapted chart at $x=0$ and
which has bounded derivatives of all positive orders.

Define $T\Omega^0$
to be the set of continuous paths $v=(v_t)_{t\in[0,1]}$ in
$T_0\R^d\cong\R^d$ with $v_0=0$ and set
\begin{equation*}
  T\Omega^{0,y}=\{v\in T\Omega^0\colon v_1=y\}\;.
\end{equation*}
Let
$\tilde\mu_\eps^0$ denote the law of the rescaled process
$(\tilde x_t^\eps)_{t\in[0,1]}$ on $T\Omega^0$ and write
$q(\eps,0,\cdot)$ for the
law of $v_1$ under the measure $\tilde\mu_\eps^0$. To obtain the
rescaled diffusion bridge measures, we disintegrate $\tilde\mu_\eps^0$
uniquely, with respect to the Lebesgue measure on $\R^d$, as
\begin{equation}\label{eq:disofmestilde}
  \tilde\mu_\eps^0(\db v)=\int_{\R^d}
  \tilde\mu_\eps^{0,y}(\db v)q(\eps,0,y)\dd y\;,
\end{equation}
where $\tilde\mu_\eps^{0,y}$ is a probability measure on $T\Omega^0$
which is supported on $T\Omega^{0,y}$, and the map
$y\mapsto\tilde\mu_\eps^{0,y}$ is weakly continuous. We can think of
$\tilde\mu_\eps^{0,y}$ as the law of the process
$(\tilde x_t^\eps)_{t\in[0,1]}$ conditioned by
$\tilde x_1^\eps=y$. In particular, this construction is consistent
with our previous definition of $\tilde\mu_\eps^{0,0}$.
Similarly, write $\tilde\mu^0$ for the law of
the limiting rescaled
diffusion process $(\tilde x_t)_{t\in[0,1]}$ on $T\Omega^0$, denote
the law of $v_1$ under $\tilde\mu^0$ by $\bar q(\cdot)$ and let
$(\tilde\mu^{0,y}\colon y\in\R^d)$ be the unique family of probability
measures we obtain by disintegrating the measure $\tilde\mu^0$ as
\begin{equation}\label{eq:disofmeslimit}
  \tilde\mu^0(\db v)=\int_{\R^d}
  \tilde\mu^{0,y}(\db v)\bar q(y)\dd y\;.
\end{equation}
To keep track of the paths of the
diffusion bridges, we fix $t_1,\dots,t_k\in(0,1)$ with $t_1<\dots<t_k$
as well as a smooth function $g$ on $(\R^d)^k$ of polynomial growth and
consider the smooth cylindrical function $G$ on
$T\Omega^0$ defined by
$G(v)=g(v_{t_1},\dots,v_{t_k})$. For $y\in\R^d$ and
$\eps>0$, set
\begin{align*}
  G_\eps(y)&=q(\eps,0,y)\int_{T\Omega^{0,y}}G(v)
  \tilde\mu_\eps^{0,y}(\db v)\qquad\mbox{and}\\
  G_0(y)  &=\bar q(y)\int_{T\Omega^{0,y}}G(v)
  \tilde\mu^{0,y}(\db v)\;.
\end{align*}
Both functions are continuous integrable functions on $\R^d$ and in
particular, we can consider their Fourier transforms $\hat
G_\eps(\xi)$ and $\hat G_0(\xi)$ given by
\begin{equation*}
  \hat G_\eps(\xi)=\int_{\R^d} G_\eps(y)\e^{\im\langle\xi,y\rangle}\dd y
  \qquad\mbox{and}\qquad
  \hat G_0(\xi)=\int_{\R^d} G_0(y)\e^{\im\langle\xi,y\rangle}\dd y\;.
\end{equation*}
Using the disintegration of measure property (\ref{eq:disofmestilde}),
we deduce that
\begin{align*}
  \hat G_\eps(\xi)&=\int_{\R^d}\int_{T\Omega^{0,y}}q(\eps,0,y)
  G(v)\tilde\mu_\eps^{0,y}(\db v)
  \e^{\im\langle\xi,y\rangle}\dd y\\
  &=\int_{T\Omega^{0}} G(v)\e^{\im\langle\xi,v_1\rangle}
    \tilde\mu_\eps^0(\db v)\\
  &=\E\left[G(\tilde x^\eps)\exp\left\{
    \im\langle\xi,\tilde x_1^\eps\rangle\right\}\right]\;.
\end{align*}
Similarly, by using (\ref{eq:disofmeslimit}), we show that
\begin{equation*}
  \hat G_0(\xi)=\E\left[G(\tilde x)\exp\left\{
  \im\langle\xi,\tilde x_1\rangle\right\}\right]\;.
\end{equation*}
We recall that $\tilde x_t^\eps\to \tilde x_t$ as $\eps\to 0$ almost
surely and in $L^p$ for all $p<\infty$. Hence,
$\hat G_\eps(\xi)\to\hat G_0(\xi)$ as $\eps\to 0$ for all
$\xi\in\R^d$. To be able to use this convergence result to
make deductions about the behaviour of the functions $G_\eps$ and
$G_0$ we need $\hat G_\eps$ to be integrable uniformly in
$\eps\in(0,1]$.
This is provided by
the following lemma, which is proven at the end of the section.
\begin{lemma}\label{lem:controloverinverse}
  For all smooth cylindrical functions $G$ on $T\Omega^0$ there are
  constants $C(G)<\infty$ such that, for all $\eps\in(0,1]$ and all
  $\xi\in\R^d$, we have
  \begin{equation}\label{eq:generalGbound}
    |\hat G_\eps(\xi)|\leq\frac{C(G)}{1+|\xi|^{d+1}}\;.
  \end{equation}
  Moreover, in the case where
  $G(v)=|v_{t_1}-v_{t_2}|^4$, there exists a constant
  $C<\infty$ such that, uniformly in $t_1,t_2\in(0,1)$, we can choose
  $C(G)=C|t_1-t_2|^2$, i.e. for all $\eps\in(0,1]$ and all
  $\xi\in\R^d$,
  \begin{equation}\label{eq:4bound}
    |\hat G_\eps(\xi)|\leq\frac{C|t_1-t_2|^2}{1+|\xi|^{d+1}}\;.
  \end{equation}
\end{lemma}
With this setup, we can prove Theorem~\ref{propn:convofbridges}.
\begin{proof}[Proof of Theorem~\ref{propn:convofbridges}]
  Applying the Fourier inversion formula and using
  (\ref{eq:generalGbound}) from Lemma~\ref{lem:controloverinverse} as
  well as the dominated convergence theorem, we deduce that
  \begin{equation}\label{Gconv}
    G_\eps(0)=\frac{1}{(2\pi)^d}\int_{\R^d}\hat G_\eps(\xi)\dd\xi
    \to
    \frac{1}{(2\pi)^d}\int_{\R^d}\hat G_0(\xi)\dd\xi=G_0(0)
    \quad\mbox{as}\quad\eps\to 0\;.
  \end{equation}
  Let $Q=\sum_{n=1}^N nd_n$ be the homogeneous dimension of the
  sub-Riemannian structure $(X_1,\dots,X_m)$. Due to the change
  of variables formula, we have
  \begin{equation*}
    q(\eps,0,y)=\eps^{Q/2}p(\eps,0,\delta_\eps(y))\;,
  \end{equation*}
  where $p$ and $q$ are the Dirichlet heat kernels,
  with respect to the Lebesgue measure on $\R^d$,
  associated to the
  processes $(x_t^1)_{t\in[0,1]}$ and $(\tilde x_t^1)_{t\in[0,1]}$,
  respectively.
  From (\ref{Gconv}), it follows that
  \begin{equation}\label{fourierconv}
    \eps^{Q/2}p(\eps,0,0)\int_{T\Omega^{0,0}}G(v)
    \tilde\mu_\eps^{0,0}(\db v)\to
    \bar q(0)\int_{T\Omega^{0,0}}G(v)
    \tilde\mu^{0,0}(\db v)
    \quad\mbox{as}\quad\eps\to 0\;.
  \end{equation}
  Choosing $g\equiv 1$ shows that
  \begin{equation}\label{heatasymp}
    \eps^{Q/2}p(\eps,0,0)\to\bar q(0)
    \quad\mbox{as}\quad\eps\to 0\;,
  \end{equation}
  which agrees with the small-time heat kernel asymptotics
  established in \cite{GBAdiag} and \cite{leandre}.
  We recall that $\bar q\colon\R^d\to [0,\infty)$ is the
  density of the random variable $\tilde x_1$, where
  $(\tilde x_t)_{t\in[0,1]}$ is the limiting rescaled process with
  generator
  \begin{equation*}
    \tilde\lo=\frac{1}{2}\sum_{i=1}^m \tilde X_i^2\;.
  \end{equation*}
  By Proposition~\ref{prop:sHlimit}, the nilpotent approximations
  $\tilde X_1,\dots,\tilde X_m$ satisfy the strong H\"ormander
  condition everywhere on $\R^d$ and since $\tilde\lo$ has vanishing
  drift, the discussions in \cite{expdecayII} imply that
  $\bar q(0)>0$. Hence, we can divide
  (\ref{fourierconv}) by (\ref{heatasymp}) to obtain
  \begin{equation*}
    \int_{T\Omega^{0,0}}G(v) \tilde\mu_\eps^{0,0}(\db v)\to
    \int_{T\Omega^{0,0}}G(v) \tilde\mu^{0,0}(\db v)
    \quad\mbox{as}\quad\eps\to 0\;.
  \end{equation*}
  Thus, the finite-dimensional distributions of $\tilde\mu_\eps^{0,0}$
  converge weakly to those of $\tilde\mu^{0,0}$ and it remains to
  establish tightness to deduce the desired convergence
  result.
  Taking $G(v)=|v_{t_1}-v_{t_2}|^4$, using the Fourier
  inversion formula and the estimate (\ref{eq:4bound}) from
  Lemma~\ref{lem:controloverinverse}, we conclude that
  \begin{equation*}
    \eps^{Q/2} p(\eps,0,0)\int_{T\Omega^{0,0}} |v_{t_1}-v_{t_2}|^4\;
    \tilde\mu_\eps^{0,0}(\db v)=G_\eps(0)\leq C|t_1-t_2|^2\;.
  \end{equation*}
  From (\ref{heatasymp}) and due to $\bar q(0)>0$,
  it further follows that there
  exists a constant
  $D<\infty$ such that, for all $t_1,t_2\in(0,1)$,
  \begin{equation*}
    \sup_{\eps\in(0,1]}\int_{T\Omega^{0,0}} |v_{t_1}-v_{t_2}|^4\;
    \tilde\mu_\eps^{0,0}(\db v)\leq D|t_1-t_2|^2\;.
  \end{equation*}
  Standard arguments finally imply that the family of laws
  $(\tilde\mu_\eps^{0,0}\colon\eps\in (0,1])$ is tight on
  $T\Omega^{0,0}$ and hence, $\tilde\mu_\eps^{0,0}\to\tilde\mu^{0,0}$
  weakly on $T\Omega^{0,0}$ as $\eps\to 0$.
\end{proof}
It remains to establish Lemma~\ref{lem:controloverinverse}. The proof
closely follows \cite[Proof of Lemma~4.1]{BMN}, where the main adjustments
needed arise due to the higher-order scaling map $\delta_\eps$.
In addition to the uniform non-degeneracy of the rescaled Malliavin
covariance matrices $\tilde c_1^\eps$, which is provided by
Theorem~\ref{thm:UND}, we need the rescaled processes
$(\tilde x_t^\eps)_{t\in[0,1]}$ and $(\tilde v_t^\eps)_{t\in[0,1]}$
defined in Section~\ref{section3dot1}
to have moments of all orders bounded uniformly in $\eps\in(0,1]$.
The latter is ensured by the following lemma.
\begin{lemma}\label{lem:tildeuvest}
  There are moment estimates of all orders for the stochastic processes
  $(\tilde x_t^\eps)_{t\in[0,1]}$ and $(\tilde v_t^\eps)_{t\in[0,1]}$
  which are
  uniform in $\eps\in(0,1]$, i.e. for all $p<\infty$, we
  have
  \begin{equation*}
    \sup_{\eps\in(0,1]}
    \E\left[\sup_{0\leq t\leq 1}|\tilde x_t^\eps|^p\right]<\infty
    \qquad\mbox{and}\qquad
    \sup_{\eps\in(0,1]}
    \E\left[\sup_{0\leq t\leq 1}|\tilde v_t^\eps|^p\right]<\infty\;.
  \end{equation*}
\end{lemma}
\begin{proof}
  We exploit the graded structure induced by the sub-Riemannian
  structure $(X_1,\dots,X_m)$ and we make use of
  the properties of an adapted chart.
  For $\tau\in[0,1]$, consider
  the It\^o stochastic differential equation in $\R^d$
  \begin{equation*}
    \db x_t^\eps(\tau)=
    \sum_{i=1}^m\tau\sqrt{\eps} X_i(x_t^\eps(\tau))\dd B_t^i+
    \tau^2\eps\underline X_0(x_t^\eps(\tau))\dd t
    \;,\quad x_0^\eps(\tau)=0
  \end{equation*}
  and let
  $\{(x_t^\eps(\tau))_{t\in[0,1]}\colon\tau\in[0,1]\}$
  be the unique family of strong solutions which is
  almost surely jointly continuous in $\tau$ and $t$. Observe
  that $x_t^\eps(0)=0$ and $x_t^\eps(1)=x_t^\eps$ for all $t\in[0,1]$,
  almost surely. Moreover, for $n\geq 1$,
  the rescaled $n$th derivative in $\tau$
  \begin{equation*}
    x_t^{\eps,(n)}(\tau)=
    \eps^{-n/2}\left(\frac{\pt}{\pt\tau}\right)^n x_t^\eps(\tau)
  \end{equation*}
  exists for all $\tau$ and $t$, almost surely. For instance,
  $(x_t^{\eps,(1)}(\tau))_{t\in[0,1]}$ is the unique
  strong solution
  of the It\^o stochastic differential equation
  \begin{align*}
    \db x_t^{\eps,(1)}(\tau)&=
    \sum_{i=1}^m X_i(x_t^\eps(\tau))\dd B_t^i+
    2\tau\sqrt{\eps}\underline X_0(x_t^\eps(\tau))\dd t\\
    &\qquad+    
    \sum_{i=1}^m\tau\sqrt{\eps}
    \nabla X_i(x_t^\eps(\tau))x_t^{\eps,(1)}(\tau)\dd B_t^i+
    \tau^2\eps\nabla\underline
    X_0(x_t^\eps(\tau))x_t^{\eps,(1)}(\tau)\dd t
    \;,\quad x_0^{\eps,(1)}(\tau)=0\;.
  \end{align*}
  In particular, we compute that
  $x_t^{\eps,(1)}(0)=\sum_{i=1}^m X_i(0)B_t^i$.
  As $\langle u,X_i(0)\rangle=0$ for all
  $i\in\{1,\dots,m\}$ and all $u\in C_1(0)^\perp$, we deduce
  \begin{equation}\label{tildex1est}
    \left\langle u,x_t^{\eps,(1)}(0)\right\rangle=0
    \quad\mbox{for all}\quad u\in C_1(0)^\perp\;.
  \end{equation}
  By looking at the corresponding stochastic differential equation for
  $(x_t^{\eps,(2)}(\tau))_{t\in[0,1]}$, we further obtain that
  \begin{equation*}
    x_t^{\eps,(2)}(0)=
    \sum_{i=1}^m \int_0^t 2\nabla X_i(0) x_s^{\eps,(1)}(0)\dd B_s^i+
    2{\underline X}_0(0)t\;.
  \end{equation*}
  Due to (\ref{tildex1est}), the only non-zero terms in
  $\nabla X_i(0) x_s^{\eps,(1)}(0)$ are scalar multiples of
  the first $d_1$ columns of $\nabla X_i(0)$, i.e. where the
  derivative is taken along a direction lying in $C_1(0)$.
  Thus, by
  property~(\ref{cond2sec3}) of an adapted chart and since
  $X_0(0)\in\operatorname{span}\{X_1(0),\dots,X_m(0)\}$, it follows that
  \begin{equation*}
    \left\langle u,x_t^{\eps,(2)}(0)\right\rangle=0
    \quad\mbox{for all}\quad u\in C_2(0)^\perp\;.
  \end{equation*}
  In general, continuing in the same way and
  by appealing to the
  Fa\`a di Bruno formula, we prove
  iteratively that, for all $n\in\{1,\dots,N-1\}$,
  \begin{equation}\label{tildexest}
    \left\langle u,x_t^{\eps,(n)}(0)\right\rangle=0
    \quad\mbox{for all}\quad u\in C_n(0)^\perp\;.
  \end{equation}
  Besides, the stochastic process
  $(x_t^\eps(\tau),x_t^{\eps,(1)}(\tau),
  \dots,x_t^{\eps,(N)}(\tau))_{t\in[0,1]}$ is the solution of a
  stochastic differential equation with graded 
  Lipschitz coefficients in the sense of Norris \cite{jamesSMC}.
  As the coefficient bounds of the graded structure are uniform in
  $\tau\in[0,1]$ and $\eps\in(0,1]$, we obtain,
  uniformly in $\tau$ and $\eps$, moment bounds
  of all orders for $(x_t^\eps(\tau),x_t^{\eps,(1)}(\tau),
  \dots,x_t^{\eps,(N)}(\tau))_{t\in[0,1]}$.
  Finally, due to~(\ref{tildexest}) we have,
  for all $n\in\{1,\dots,N\}$ and all
  $u\in C_n(0)\cap C_{n-1}(0)^\perp$,
  \begin{equation*}
    \left\langle u,\tilde x_t^\eps\right\rangle=
    \left\langle u,\eps^{-n/2} x_t^\eps\right\rangle=
    \left\langle u,
      \int_0^1\int_0^{\tau_1}\dots\int_0^{\tau_{n-1}}x_t^{\eps,(n)}(\tau_n)
      \dd\tau_n\dd\tau_{n-1}\dots\dd\tau_1
      \right\rangle\;.
  \end{equation*}
  This together with the uniform moment bounds implies
  the claimed result that, for all $p<\infty$,
  \begin{equation*}
    \sup_{\eps\in(0,1]}
    \E\left[\sup_{0\leq t\leq 1}|\tilde x_t^\eps|^p\right]<\infty\;.
  \end{equation*}
  We proceed similarly to establish the second estimate. Let
  $\{(v_t^\eps(\tau))_{t\in[0,1]}\colon\tau\in[0,1]\}$ be the
  unique family of strong solutions to the It\^o stochastic
  differential equation in $\R^d$
  \begin{equation*}
    \db v_t^\eps(\tau)=
    -\sum_{i=1}^m\tau\sqrt{\eps}v_t^\eps(\tau)
    \nabla X_i(x_t^\eps(\tau))\dd B_t^i
    -\tau^2\eps v_t^\eps(\tau)\left(\nabla{\underline X}_0-
      \sum_{i=1}^m(\nabla X_i)^2\right)(x_t^\eps(\tau))\dd t\;,
    \quad v_0^\eps(\tau)=I
  \end{equation*}
  which is almost surely jointly continuous in $\tau$ and $t$. We
  note that $v_t^\eps(0)=I$ and $v_t^\eps(1)=v_t^\eps$
  for all $t\in[0,1]$, almost surely. For $n\geq 1$, set
  \begin{equation*}
    v_t^{\eps,(n)}(\tau)=
    \eps^{-n/2}\left(\frac{\pt}{\pt\tau}\right)^n v_t^\eps(\tau)\;,
  \end{equation*}
  which exists
  for all $\tau$ and $t$, almost surely. For
  $n_1,n_2\in\{1,\dots,N\}$ and
  $u^1\in C_{n_1}(0)\cap C_{n_1-1}(0)^\perp$
  as well as $u^2\in C_{n_2}(0)\cap C_{n_2-1}(0)^\perp$, we have
  \begin{equation*}
    \left\langle u^1,\tilde v_t^\eps u^2\right\rangle=
    \eps^{-(n_1-n_2)/2}\left\langle u^1,v_t^\eps u^2\right\rangle\;.
  \end{equation*}
  Therefore, if $n_1\leq n_2$, we obtain the bound
  $|\langle u^1,\tilde v_t^\eps u^2\rangle|\leq
    |\langle u^1,v_t^\eps u^2\rangle|$. On the other hand, if
  $n_1>n_2$ then $\langle u^1,u^2\rangle=0$ and in a similar way to
  proving
  (\ref{tildexest}), we show that
  \begin{equation*}
    \left\langle u^1,v_t^{\eps,(k)}(0)u^2\right\rangle=0
    \quad\mbox{for all}\quad k\in\{1,\dots, n_1-n_2-1\}
  \end{equation*}
  by repeatedly using property~(\ref{cond2sec3}) of an adapted chart.
  This allows us to write
  \begin{equation*}
    \left\langle u^1,\tilde v_t^\eps u^2\right\rangle=
    \left\langle u^1,
      \left(\int_0^1\int_0^{\tau_1}\dots\int_0^{\tau_{n_1-n_2-1}}
        v_t^{\eps,(n_1-n_2)}(\tau_{n_1-n_2})
      \dd\tau_{n_1-n_2}\dd\tau_{n_1-n_2-1}\dots\dd\tau_1\right) u^2
      \right\rangle
  \end{equation*}
  for $n_1>n_2$. As the stochastic process
  $(x_t^\eps(\tau), v_t^\eps(\tau),
      x_t^{\eps,(1)}(\tau),v_t^{\eps,(1)}(\tau),\dots,
      x_t^{\eps,(N)}(\tau),v_t^{\eps,(N)}(\tau))_{t\in[0,1]}$
  is the solution of a stochastic differential equation with graded
  Lipschitz coefficients in the sense of Norris~\cite{jamesSMC}, with
  the coefficient bounds of the graded structure being uniform in
  $\tau\in[0,1]$ and $\eps\in(0,1]$, the second
  result claimed follows.
\end{proof}
We finally present the proof of
Lemma~\ref{lem:controloverinverse}.
For some of the technical arguments which carry over unchanged,
we simply refer the reader to \cite{BMN}.
\begin{proof}[Proof of Lemma~\ref{lem:controloverinverse}]
  Let $(x_t^\eps)_{t\in[0,1]}$ be the process in $\R^d$ and
  $(u_t^\eps)_{t\in[0,1]}$ as well as $(v_t^\eps)_{t\in[0,1]}$ be the
  processes in $\R^d\otimes(\R^d)^*$ which are defined
  as the unique strong solutions of the following
  system of It\^o stochastic differential equations.
  \begin{align}
    \label{SDE4x}
    \db x_t^\eps&=\sum_{i=1}^m\sqrt{\eps} X_i(x_t^\eps)\dd B_t^i
    +\eps{\underline X}_0(x_t^\eps)\dd t\;,\quad x_0^\eps=0\\
    \nonumber
    \db u_t^\eps&=\sum_{i=1}^m\sqrt{\eps}\nabla X_i(x_t^\eps)u_t^\eps\dd B_t^i
    +\eps \nabla{\underline X}_0(x_t^\eps)u_t^\eps\dd t\;,
    \quad u_0^\eps=I\\
    \nonumber
    \db v_t^\eps&=-\sum_{i=1}^m\sqrt{\eps}v_t^\eps\nabla X_i(x_t^\eps)\dd B_t^i
    -\eps v_t^\eps\left(\nabla{\underline X}_0-
      \sum_{i=1}^m(\nabla X_i)^2\right)(x_t^\eps)\dd t\;,
    \quad v_0^\eps=I
  \end{align}
  Fix $k\in\{1,\dots,d\}$. For $\eta\in\R^d$, consider the perturbed
  process $(B_t^{\eta})_{t\in[0,1]}$ in $\R^m$ given by
  \begin{equation*}
    \db B_t^{\eta,i}=\db B_t^i+\eta\left(\sqrt{\eps}\scale_\eps^{-1}
      \left(v_t^\eps X_i(x_t^\eps)\right)\right)^k\dd t\;,
    \quad B_0^\eta=0\;,
  \end{equation*}
  where $(\sqrt{\eps}\scale_\eps^{-1}
    \left(v_t^\eps X_i(x_t^\eps)\right))^k$ denotes the $k$th
  component of the vector $\sqrt{\eps}\scale_\eps^{-1}
  \left(v_t^\eps X_i(x_t^\eps)\right)$ in $\R^d$.
  Write $(x_t^{\eps,\eta})_{t\in[0,1]}$ for
  the strong solution of the stochastic differential equation
  (\ref{SDE4x}) with the driving Brownian motion
  $(B_t)_{t\in[0,1]}$ replaced by $(B_t^{\eta})_{t\in[0,1]}$.
  We choose a version of the family of processes 
  $(x_t^{\eps,\eta})_{t\in[0,1]}$ which is almost surely smooth in $\eta$
  and set
  \begin{equation*}
    \left((x^\eps)'_t\right)^k=
    \left.\frac{\pt}{\pt\eta}\right|_{\eta=0}x_t^{\eps,\eta}\;.
  \end{equation*}
  The derived process
  $((x^\eps)'_t)_{t\in[0,1]}=(\left((x^\eps)'_t\right)^1,
  \dots,\left((x^\eps)'_t\right)^d)_{t\in[0,1]}$ in $\R^d\otimes\R^d$
  associated with the process $(x_t^\eps)_{t\in[0,1]}$
  then satisfies the It\^o stochastic
  differential equation
  \begin{equation*}
    \db (x^\eps)_t'=\sum_{i=1}^m\sqrt{\eps}\nabla
    X_i(x_t^\eps)(x^\eps)_t'\dd B_t^i+\eps\nabla
    {\underline X}_0(x_t^\eps)(x^\eps)_t'\dd t
    +\sum_{i=1}^m\sqrt{\eps}X_i(x_t^\eps)\otimes
    \left(\sqrt{\eps}\scale_{\eps}^{-1}\left(v_t^\eps
        X_i(x_t^\eps)\right)\right)\db t
  \end{equation*}
  subject to $(x^\eps)'_0=0$. Using the expression
  (\ref{eq:rescaledMCM}) for the rescaled Malliavin
  covariance matrix $\tilde c_t^\eps$, we show that
  $(x^\eps)'_t=u_t^\eps\scalemat_\eps \tilde c_t^\eps$.
  It follows that for the derived process
  $((\tilde x^\eps)'_t)_{t\in[0,1]}$ associated with the rescaled process
  $(\tilde x^\eps_t)_{t\in[0,1]}$ and the stochastic process $(\tilde
  u_t^\eps)_{t\in[0,1]}$ given by $\tilde u_t^\eps=\scalemat_\eps^{-1}
  u_t^\eps\scalemat_\eps$, we have
  \begin{equation*}
    (\tilde x^\eps)'_t=\tilde u_t^\eps\tilde c_t^\eps\;.
  \end{equation*}
  Note that both $\tilde u_1^\eps$ and $\tilde c_1^\eps$ are invertible
  for all $\eps>0$ with $(\tilde u_1^\eps)^{-1}=\tilde v_1^\eps$.
  Let $(r_t^\eps)_{t\in[0,1]}$ be the process defined by
  \begin{equation*}
    \db r_t^\eps=\sum_{i=1}^m\sqrt{\eps}\scale_\eps^{-1}
      \left(v_t^\eps X_i(x_t^\eps)\right)\db B_t^i\;,\quad r_0^\eps=0
  \end{equation*}
  and set $y_t^{\eps,(0)}=(x_{t\wedge t_1}^\eps,\dots,x_{t\wedge
    t_k}^\eps,x_t^\eps,v_t^\eps,r_t^\eps,(x^\eps)'_t)$.
  The underlying graded Lipschitz structure, in the sense of
  Norris~\cite{jamesSMC},
  allows us, for $n\geq 0$, to recursively define
  \begin{equation*}
    z_t^{\eps,(n)}=\left(y_t^{\eps,(0)},\dots,y_t^{\eps,(n)}\right)
  \end{equation*}
  by first solving for the derived process
  $((z^{\eps,(n)})'_t)_{t\in[0,1]}$, then writing
  \begin{equation*}
    \left(z^{\eps,(n)}\right)'_t=\left(\left(y^{\eps,(0)}\right)'_t,
      \dots,\left(y^{\eps,(n)}\right)'_t\right)
  \end{equation*}
  and finally setting
  $y_t^{\eps,(n+1)}=(y^{\eps,(n)})'_t$\;.
  
  Consider the random variable $y^\eps=((\tilde x^\eps)'_1)^{-1}$ in
  $(\R^d)^*\otimes (\R^d)^*$ and let
  $\phi=\phi(y^\eps,z_1^{\eps,(n)})$ be a polynomial in $y^\eps$, where
  the coefficients are continuously
  differentiable in $z_1^{\eps,(n)}$ and of polynomial growth, along
  with their derivatives. Going through the deductions made from
  Bismut's integration by parts formula in \cite[Proof of
  Lemma~4.1]{BMN} with $R\equiv 0$ and $F\equiv 0$
  shows that for any continuously
  differentiable, bounded function $f\colon\R^d\to\R$
  with bounded first derivatives
  and any $k\in\{1,\dots,d\}$, we have
  \begin{equation*}
    \E\left[\nabla_k f(\tilde x_1^\eps)
      \phi\left(y^\eps,z_1^{\eps,(n)}\right)\right]=
    \E\left[f(\tilde x_1^\eps)\nabla_k^*\phi
      \left(y^\eps,z_1^{\eps,(n+1)}\right)\right]\;,
  \end{equation*}
  where
  \begin{align*}
    \nabla_k^*\phi\left(y^\eps,z_1^{\eps,(n+1)}\right)
    \qquad\qquad\qquad\qquad\qquad\qquad\quad\\
    =\tau_k\left(y^\eps\otimes r_1^\eps+y^\eps (\tilde x^\eps)''_1
      y^\eps\right)\phi\left(y^\eps,z_1^{\eps,(n)}\right)
    &+\tau_k\left(y^\eps\otimes\left(\nabla_y
        \phi\left(y^\eps,z_1^{\eps,(n)}\right)
        y^\eps (\tilde x^\eps)''_1 y^\eps\right)\right)\\
    &-\tau_k\left(y^\eps\otimes\left(\nabla_z
        \phi\left(y^\eps,z_1^{\eps,(n)}\right)
        \left(z^{\eps,(n)}\right)'_1\right)\right)\;,
  \end{align*}
  and $\tau_k\colon(\R^d)^*\otimes (\R^d)^*\otimes\R^d\to\R$ is the
  linear map given by
  $\tau_k(e_l^*\otimes e_{k'}^*\otimes e_{l'})
   =\delta_{k k'}\delta_{l l'}\,.$
  Starting from
  \begin{equation*}
    \phi\left(y^\eps,z_1^{\eps,(0)}\right)=G(\tilde x^\eps)
    =g\left(\tilde x_{t_1}^\eps,\dots,\tilde x_{t_k}^\eps\right)
  \end{equation*}
  we see inductively that, for any multi-index
  $\alpha=(k_1,\dots,k_n)$,
  \begin{equation*}
    \E\left[\nabla^\alpha f(\tilde x_1^\eps)G(\tilde x^\eps)\right]=
    \E\left[f(\tilde x_1^\eps)(\nabla^*)^\alpha G
      \left(y^\eps,z_1^{\eps,(n)}\right)\right]\;.
  \end{equation*}
  Fixing $\xi\in\R^d$ and
  choosing $f(\cdot)=\e^{\im\langle\xi,\cdot\rangle}$ in this
  integration by parts formula yields
  \begin{equation*}
    |\xi^\alpha||\hat G_\eps(\xi)|\leq
    \E\left[\left|(\nabla^*)^\alpha G\left(y^\eps,z_1^{\eps,(n)}
        \right)\right|\right]\;.
  \end{equation*}
  In order to deduce the bound (\ref{eq:generalGbound}), it remains to
  establish that $C_\eps(\alpha,G)=\E[|(\nabla^*)^\alpha
  G(y^\eps,z_1^{\eps,(n)})|]$ can be controlled uniformly in
  $\eps$. Due to $y^\eps=(\tilde c_1^\eps)^{-1} \tilde v_1^\eps$,
  Theorem~\ref{thm:UND} and the second estimate from
  Lemma~\ref{lem:tildeuvest} immediately imply that,
  for all $p<\infty$,
  \begin{equation}\label{ymoments}
    \sup_{\eps\in(0,1]}\E\left[\left|y^\eps\right|^p\right]<\infty\;.
  \end{equation}
  Moreover, from the first moment estimate in
  Lemma~\ref{lem:tildeuvest}, it follows that all processes derived from
  the rescaled process
  $(\tilde x_t^\eps)_{t\in[0,1]}$ have moments of all orders
  bounded uniformly in $\eps\in(0,1]$.
  Similarly, for $n=d+1$ and all $p<\infty$, we obtain
  \begin{equation}\label{zmoments}
    \sup_{\eps\in(0,1]}\E\left[\left|z_1^{\eps,(n)}\right|^p\right]<\infty\;,
  \end{equation}
  where we observe that, for all $n\in\{0,1,\dots,N-1\}$ and all
  $u\in C_{n+1}(0)\cap C_n(0)^\perp$,
  \begin{equation*}
    \left\langle u,r_t^\eps\right\rangle=
    \sum_{i=1}^m\int_0^t\left\langle u,
      \eps^{-n/2}v_s^\eps X_i(x_s^\eps)\right\rangle\dd B_s^i\;,
  \end{equation*}
  and use Lemma~\ref{lem:leadterm} to show that there is no
  singularity in the process $(r_t^\eps)_{t\in[0,1]}$ as $\eps\to 0$.
  Since $(\nabla^*)^\alpha G$ is of polynomial growth in
  the argument
  $(y^\eps,z_1^{\eps,(n)})$, the moment
  estimates (\ref{ymoments}) and (\ref{zmoments})
  show that
  $C_\eps(\alpha,G)$ is bounded uniformly in $\eps\in(0,1]$.
  This establishes (\ref{eq:generalGbound}).

  Finally, the same proof as presented in \cite[Proof of
  Lemma~4.1]{BMN} shows that we
  have (\ref{eq:4bound}) in the special case
  where $G(v)=|v_{t_1}-v_{t_2}|^4$ for some $t_1,t_2\in(0,1)$.
  Let the process
  $(\tilde x_t^{\eps,(0)})_{t\in[0,1]}$ be given by
  $\tilde x_t^{\eps,(0)}=\tilde x_t^\eps$ and, recursively for $n\geq
  0$, define $(\tilde x_t^{\eps,(n+1)})_{t\in[0,1]}$ by
  $\tilde x_t^{\eps,(n+1)}=(\tilde x_t^\eps,(\tilde x^{\eps,(n)})'_t)$.
  Then, for all $p\in[1,\infty)$, there exists a constant
  $D(p)<\infty$ such that,
  uniformly in $t_1,t_2\in(0,1)$ and in $\eps\in(0,1]$,
  \begin{equation*}
    \E\left[\left|\tilde x_{t_1}^{\eps,(n)}-
                  \tilde x_{t_2}^{\eps,(n)}\right|^{4p}\right]
              \leq D(p)|t_1-t_2|^{2p}\;.
  \end{equation*}
  Furthermore, from the expression for the adjoint operator
  $\nabla_k^*$ we deduce
  that, for all $n\geq 1$ and any multi-index
  $\alpha=(k_1,\dots,k_n)$, there exists a random variable $M_\alpha$,
  with moments of all orders which are bounded uniformly in
  $\eps\in(0,1]$, such that
  \begin{equation*}
    \left(\nabla^*\right)^\alpha G\left(y^\eps,z_1^{\eps,(n)}\right) =
    M_\alpha\left|\tilde x_{t_1}^{\eps,(n)}-\tilde x_{t_2}^{\eps,(n)}\right|^{4}\;.
  \end{equation*}
  By using H\"older's inequality, we conclude that there exists a
  constant $C(\alpha)<\infty$ such that, uniformly in
  $t_1,t_2\in(0,1)$ and $\eps\in(0,1]$, we obtain
  \begin{equation*}
    C_\eps(\alpha,G)\leq C(\alpha)|t_1-t_2|^2\;,
  \end{equation*}
  which implies (\ref{eq:4bound}).
\end{proof}

\section{Localisation argument}\label{pp:local}
In proving Theorem~\ref{thm:convofbridges} by localising
Theorem~\ref{propn:convofbridges}, we use the same
localisation argument as presented in \cite[Section~5]{BMN}. This is
possible due to \cite[Theorem~6.1]{BMN}, which provides a control over
the amount of heat diffusing between two fixed points without leaving
a fixed closed subset, also covering the diagonal case.
After the proof, we give an example to illustrate
Theorem~\ref{thm:convofbridges} 
and we remark on deductions made for
the $\sqrt{\eps}$-rescaled fluctuations of diffusion loops.

Let $\lo$ be a differential operator on $M$
satisfying the conditions of Theorem~\ref{thm:convofbridges}
and let $(X_1,\dots,X_m)$ be
a sub-Riemannian structure for the diffusivity of $\lo$. Define
$X_0$ to be the smooth vector field on $M$ given by requiring
\begin{equation*}
  \lo=\frac{1}{2}\sum_{i=1}^m X_i^2+X_0
\end{equation*}
and recall that
$X_0(y)\in\operatorname{span}\{X_1(y),\dots,X_m(y)\}$
for all $y\in M$.
Let $(U_0,\theta)$ be an adapted chart to the filtration induced by
$(X_1,\dots,X_m)$ at $x\in M$ and extend it to a smooth map
$\theta\colon M\to\R^d$. By passing to a smaller set if necessary, we
may assume that the closure of $U_0$ is compact. Let $U$ be a domain
in $M$ containing $x$ and compactly contained in $U_0$.
We start by
constructing a
differential operator $\bar\lo$ on $\R^d$ which satisfies the
assumptions of Theorem~\ref{propn:convofbridges}
with the identity map being an adapted chart at $0$ and such that
$\lo(f)=\bar\lo(f\circ\theta^{-1})\circ\theta$ for all
$f\in C^\infty(U)$.

Set $V=\theta(U)$ and $V_0=\theta(U_0)$. Let $\chi$ be a smooth
function on $\R^d$ which satisfies $\ind_V\leq\chi\leq \ind$ and
where $\{\chi>0\}$ is compactly contained in
$V_0$. The existence of such a function is always guaranteed.
Besides, we pick another smooth function $\rho$ on $\R^d$ with
$\ind_{V}\leq \ind-\rho\leq\ind_{V_0}$ and such that
$\chi+\rho$ is everywhere positive.
Define vector fields
$\bar X_0,\bar X_1,\dots,\bar X_m,\bar X_{m+1},\dots,\bar X_{m+d}$
on $\R^d$ by
\begin{align*}
  \bar X_i(z)&=
  \begin{cases}
    \chi(z)\dd\theta_{\theta^{-1}(z)}
    \left(X_i\left(\theta^{-1}(z)\right)\right) & \mbox{if }z\in V_0\\
    0 & \mbox{if }z\in\R^d\setminus V_0
  \end{cases}
  &&\mbox{for } i\in\{0,1,\dots,m\}\;,\\
  \bar X_{m+k}(z)&=\rho(z)e_k &&\mbox{for } k\in\{1,\dots,d\}\;,
\end{align*}
where $e_1,\dots,e_d$ is the standard basis in $\R^d$.
We note that
$X_0(y)\in\operatorname{span}\{X_1(y),\dots,X_m(y)\}$
for all $y\in M$ implies that 
$\bar X_0(z)\in\operatorname{span}\{\bar X_1(z),\dots,\bar X_m(z)\}$
holds for all $z\in\R^d$. Moreover, the vector fields
$\bar X_1,\dots,\bar X_m$ satisfy the strong H\"ormander
condition on $\{\chi>0\}$, while
$\bar X_{m+1},\dots,\bar X_{m+d}$ themselves
span $\R^d$ on $\{\rho>0\}$. As $U_0$ is assumed to have compact
closure, the vector fields constructed are
all bounded with bounded derivatives of all orders.
Hence, the differential operator $\bar\lo$ on $\R^d$ given by
\begin{equation*}
  \bar\lo=\frac{1}{2}\sum_{i=1}^{m+d} \bar X_i^2 +\bar X_0
\end{equation*}
satisfies the assumptions of Theorem~\ref{propn:convofbridges}.
We further observe that, on $V$,
\begin{equation*}
  \bar X_i = \theta_*(X_i)\quad\mbox{for all }
  i\in\{0,1,\dots,m\}\;,
\end{equation*}
which yields the the desired property that
$\bar\lo=\theta_*\lo$ on $V$.
Additionally, we see that the nilpotent approximations of
$(\bar X_1,\dots,\bar X_m,\bar X_{m+1},\dots,\bar X_{m+d})$
are $(\tilde X_1,\dots,\tilde X_m,0,\dots,0)$ which shows that the
limiting rescaled processes on $\R^d$
associated to the processes with generator
$\eps\bar\lo$ and $\eps\lo$, respectively, have the same generator
$\tilde\lo$.
Since $(U_0,\theta)$, and in particular the restriction
$(U,\theta)$ is an adapted chart at $x$, it also follows that
the identity map on $\R^d$ is an adapted chart to the
filtration induced by the sub-Riemannian structure
$(\bar X_1,\dots,\bar X_m,\bar X_{m+1},\dots,\bar X_{m+d})$ on $\R^d$
at $0$.
Thus, Theorem~\ref{propn:convofbridges} holds with the identity map
as the global diffeomorphism
and we associate the same anisotropic dilation
$\delta_\eps\colon\R^d\to\R^d$ with the adapted charts
$(U,\theta)$ at $x$ and $(V,I)$ at $0$.
We use this to finally prove our main result.
\begin{proof}[Proof of Theorem~\ref{thm:convofbridges}]
  Let $\bar p$ be the Dirichlet heat kernel for $\bar\lo$ with respect
  to the Lebesgue measure $\lambda$ on $\R^d$. Choose a positive
  smooth measure $\nu$ on $M$ which satisfies
  $\nu=(\theta^{-1})_*\lambda$ on $U$ and let $p$ denote the Dirichlet
  heat kernel for $\lo$ with respect to $\nu$.
  Write $\mu_\eps^{0,0,\R^d}$ for the diffusion loop measure on
  $\Omega^{0,0}(\R^d)$ associated with the operator $\eps\bar\lo$
  and write $\tilde\mu_\eps^{0,0,\R^d}$ for the rescaled loop measure
  on $T\Omega^{0,0}(\R^d)$, which is the image measure of
  $\mu_\eps^{0,0,\R^d}$ under the scaling map
  $\bar\sigma_\eps\colon\Omega^{0,0}(\R^d)\to T\Omega^{0,0}(\R^d)$
  given by
  \begin{equation*}
    \bar\sigma_\eps(\omega)_t=\delta_\eps^{-1}\left(\omega_t\right)\;.
  \end{equation*}
  Moreover, let $\tilde\mu^{0,0,\R^d}$ be the loop measure
  on $T\Omega^{0,0}(\R^d)$ associated with the
  stochastic process $(\tilde x_t)_{t\in[0,1]}$ on $\R^d$
  starting from $0$ and
  having generator $\tilde\lo$ and let $\bar q$ denote the probability
  density function of $\tilde x_1$.
  From Theorem~\ref{propn:convofbridges}, we know that
  $\tilde\mu_\eps^{0,0,\R^d}$ converges weakly to
  $\tilde\mu^{0,0,\R^d}$ on $T\Omega^{0,0}(\R^d)$ as $\eps\to 0$, and
  its proof also shows that
  \begin{equation}\label{heatest}
    \bar p(\eps,0,0)=\eps^{-Q/2}\bar q(0)(1+o(1))
    \quad\mbox{as}\quad\eps\to 0\;.
  \end{equation}
  Let $p_U$ denote the Dirichlet heat kernel in $U$ of the
  restriction of $\lo$ to $U$ and write $\mu_\eps^{x,x,U}$ for the
  diffusion bridge measure on $\Omega^{x,x}(U)$ associated with the
  restriction of the operator $\eps\lo$ to $U$.
  For any measurable set
  $A\subset\Omega^{x,x}(M)$, we have
  \begin{equation}\label{locest1}
    p(\eps,x,x)\mu_\eps^{x,x}(A) =
    p_U(\eps,x,x)\mu_\eps^{x,x,U}(A\cap\Omega^{x,x}(U))+
    p(\eps,x,x)\mu_\eps^{x,x}(A\setminus\Omega^{x,x}(U))\;.
  \end{equation}
  Additionally, by counting paths and since
  $\nu=(\theta^{-1})_*\lambda$ on $U$, we obtain
  \begin{equation}\label{locest2}
    \bar p(\eps,0,0)
    \mu_\eps^{0,0,\R^d}\left(\theta(A\cap\Omega^{x,x}(U))\right)
    = p_U(\eps,x,x)
      \mu_\eps^{x,x,U}\left(A\cap\Omega^{x,x}(U)\right)\;,
  \end{equation}
  where $\theta(A\cap\Omega^{x,x}(U))$ denotes the subset
  $\{(\theta(\omega_t))_{t\in[0,1]}\colon\omega\in
  A\cap\Omega^{x,x}(U)\}$ of $\Omega^{0,0}(\R^d)$.
  Let $B$ be a
  bounded measurable subset of the set $T\Omega^{x,x}(M)$ of
  continuous paths $v=(v_t)_{t\in[0,1]}$ in $T_xM$ with $v_0=0$ and
  $v_1=0$.
  For $\eps>0$ sufficiently small, we have
  $\sigma_\eps^{-1}(B)\subset\Omega^{x,x}(U)$ and so
  (\ref{locest1}) and (\ref{locest2}) imply that
  \begin{equation*}
    p(\eps,x,x)\mu_\eps^{x,x}\left(\sigma_\eps^{-1}(B)\right)=
    \bar p(\eps,0,0)
    \mu_\eps^{0,0,\R^d}\left(\theta\left(\sigma_\eps^{-1}(B)\right)\right)\;.
  \end{equation*}
  Since $\mu_\eps^{x,x}(\sigma_\eps^{-1}(B))=\tilde\mu_\eps^{x,x}(B)$
  and
  \begin{equation*}
    \mu_\eps^{0,0,\R^d}\left(\theta\left(\sigma_\eps^{-1}(B)\right)\right)=
    \mu_\eps^{0,0,\R^d}\left(\bar\sigma_\eps^{-1}(\db\theta_x(B))\right)=
    \tilde\mu_\eps^{0,0,\R^d}(\db\theta_x(B))\;,
  \end{equation*}
  we established that
  \begin{equation}\label{rescaleest}
    p(\eps,x,x)\tilde\mu_\eps^{x,x}(B)=
    \bar p(\eps,0,0)\tilde\mu_\eps^{0,0,\R^d}(\db\theta_x(B))\;.
  \end{equation}
  Moreover, it holds true that
  $\mu_\eps^{0,0,\R^d}(\theta(\Omega^{x,x}(U))\to 1$ as $\eps\to 0$.
  Therefore, taking $A=\Omega^{x,x}(M)$ in
  (\ref{locest2}) and using (\ref{heatest}) gives
  \begin{equation*}
    p_U(\eps,x,x)=\eps^{-Q/2}\bar q(0)(1+o(1))
    \quad\mbox{as}\quad\eps\to 0\;.
  \end{equation*}
  By \cite[Theorem~6.1]{BMN}, we know that
  \begin{equation*}
    \limsup_{\eps\to 0}\eps\log p(\eps,x,M\setminus U,x)
    \leq -\frac{d(x,M\setminus U,x)^2}{2}\;,
  \end{equation*}
  where $p(\eps,x,M\setminus U,x)=p(\eps,x,x)-p_U(\eps,x,x)$ and
  $d(x,M\setminus U,x)$ is the sub-Riemannian distance from $x$ to $x$
  through $M\setminus U$.
  Since $d(x,M\setminus U,x)$ is strictly
  positive, it follows that
  \begin{equation*}
    p(\eps,x,x)=p_U(\eps,x,x)+p(\eps,x,M\setminus U,x)=
    \eps^{-Q/2}\bar q(0)(1+o(1))
    \quad\mbox{as}\quad\eps\to 0\;.
  \end{equation*}
  Hence, due to (\ref{rescaleest}), we have
  $\tilde\mu_\eps^{x,x}(B)=\tilde\mu_\eps^{0,0,\R^d}(\db\theta_x(B))(1+o(1))$
  for any bounded measurable set $B\subset T\Omega^{x,x}(M)$.
  From the weak convergence of $\tilde\mu_\eps^{0,0,\R^d}$ to
  $\tilde\mu^{0,0,\R^d}$ on $T\Omega^{0,0}(\R^d)$ as $\eps\to 0$ and
  since $\tilde\mu^{0,0,\R^d}(\db\theta_x(B))=\tilde\mu^{x,x}(B)$, we
  conclude that the diffusion loop measures
  $\tilde\mu_\eps^{x,x}$ converge weakly to the loop measure
  $\tilde\mu^{x,x}$ on $T\Omega^{0,0}(M)$ as $\eps\to 0$.
\end{proof}
We close with an example and a remark.
\begin{eg0}\label{example2}\rm
  Consider the same setup as in Example~\ref{example},
  i.e. $M=\R^2$ with $x=0$ fixed and the vector fields
  $X_1, X_2$ on $\R^2$ defined by
  \begin{equation*}
    X_1=\frac{\pt}{\pt x^1}+x^1\frac{\pt}{\pt x^2}
    \qquad\mbox{and}\qquad X_2=x^1\frac{\pt}{\pt x^1}
  \end{equation*}
  in Cartesian coordinates $(x^1,x^2)$.
  We recall that these
  coordinates are not adapted to the filtration induced by
  $(X_1,X_2)$ at $0$
  and we start off by illustrating why this chart is not
  suitable for our analysis. The unique strong solution
  $(x_t^\eps)_{t\in[0,1]}=(x_t^{\eps,1},x_t^{\eps,2})_{t\in[0,1]}$
  of the Stratonovich stochastic differential equation in $\R^2$
  \begin{align*}
    \stb x_t^{\eps,1} &= \sqrt{\eps}\std B_t^1+
                       \sqrt{\eps}x_t^{\eps,1}\std B_t^2\\
    \stb x_t^{\eps,2} &= \sqrt{\eps}x_t^{\eps,1}\std B_t^1
  \end{align*}
  subject to $x_0^\eps=0$ is given by
  \begin{equation*}
    x_t^\eps=
    \left(\sqrt{\eps}\int_0^t\e^{\sqrt{\eps}\left(B_t^2-B_s^2\right)}\std B_s^1,
      \eps\int_0^t\left(
      \int_0^s\e^{\sqrt{\eps}\left(B_s^2-B_r^2\right)}\std
      B_r^1\right)\std B_s^1\right)\;.
  \end{equation*}
  Even though the step of the filtration induced by $(X_1,X_2)$ at
  $0$ is $N=3$, rescaling the stochastic
  process $(x_t^\eps)_{t\in[0,1]}$ by $\eps^{-3/2}$ in any direction
  leads to a blow-up in the limit $\eps\to 0$. Instead, the
  highest-order rescaled process we can consider is
  $(\eps^{-1/2}x_t^{\eps,1},\eps^{-1}x_t^{\eps,2})_{t\in[0,1]}$ whose
  limiting process, as $\eps\to 0$, is characterised by
  \begin{equation*}
    \lim_{\eps\to 0}\left(\eps^{-1/2}x_t^{\eps,1},\eps^{-1}x_t^{\eps,2}\right)=
    \left(B_t^1,\frac{1}{2}\left(B_t^1\right)^2\right)\;.
  \end{equation*}
  Thus, these rescaled processes localise around a parabola in $\R^2$.
  As the Malliavin covariance matrix of
  $(B_1^1,\frac{1}{2}(B_1^1)^2)$ is degenerate, the Fourier
  transform argument from Section~\ref{sec:convofdiffmeas}
  cannot be used.
  Rather, we
  first need to apply an additional rescaling along the
  parabola to recover a non-degenerate limiting process.
  This is the reason why we choose to
  work in an adapted chart because it allows
  us to express the overall rescaling needed
  as an anisotropic dilation.

  Let $\theta\colon\R^2\to\R^2$ be the same global adapted chart as
  used in Example~\ref{example} and let
  $\scale_\eps\colon\R^2\to\R^2$ be the
  associated anisotropic dilation. We showed that the nilpotent
  approximations $\tilde X_1, \tilde X_2$ of the vector fields
  $X_1, X_2$ are
  \begin{equation*}
    \tilde X_1=\frac{\pt}{\pt y^1}
    \qquad\mbox{and}\qquad
    \tilde X_2=-\left(y^1\right)^2\frac{\pt}{\pt y^2}\;,
  \end{equation*}
  with respect to Cartesian coordinates $(y^1,y^2)$ on the second copy
  of $\R^2$. From the convergence result (\ref{convergence}), it
  follows that, for all $t\in[0,1]$,
  \begin{equation*}
    \scale_\eps^{-1}\left(\theta(x_t^\eps)\right)\to
    \left(B_t^1,-\int_0^t\left(B_s^1\right)^2\std B_s^2\right)
    \quad\mbox{as}\quad\eps\to 0\;.
  \end{equation*}
  Since $\db\theta_0\colon\R^2\to\R^2$ is the identity,
  Theorem~\ref{thm:convofbridges} says that the suitably rescaled
  fluctuations of the diffusion loop at $0$ associated to the
  stochastic process with generator $\lo=\frac{1}{2}(X_1^2+X_2^2)$
  converge weakly to the loop obtained by conditioning
  $(B_t^1,-\int_0^t(B_s^1)^2\std B_s^2)_{t\in[0,1]}$
  to return to $0$ at time $1$.
\end{eg0}
\begin{remark}
  We show that Theorem~\ref{thm:convofbridges} and
  Theorem~\ref{propn:convofbridges}
  allow us to make deductions about
  the $\sqrt{\eps}$-rescaled fluctuations of diffusion loops.
  For the rescaling map $\tau_\eps\colon\Omega^{x,x}\to T\Omega^{0,0}$
  given by
  \begin{equation*}
    \tau_\eps(\omega)_t=(\db\theta_x)^{-1}
  \left(\eps^{-1/2}\theta(\omega_t)\right)\;,
  \end{equation*}
  we are interested in the behaviour of the measures
  $\mu_\eps^{x,x}\circ\tau_\eps^{-1}$ in the limit $\eps\to 0$.
  Let $e_1,\dots,e_d$ be the standard basis in $\R^d$
  and define
  $\psi\colon T\Omega^{0,0}\to T\Omega^{0,0}$ by
  \begin{equation*}
    \psi(v)_t=\sum_{i=1}^{d_1}
    \left\langle\db\theta_x(v_t),e_i\right\rangle
    \left(\db\theta_x\right)^{-1}e_i\;.
  \end{equation*}
  The map $\psi$ takes a path in $T\Omega^{0,0}$ and projects it onto
  the component living in the subspace $C_1(x)$ of $T_xM$.
  Since the maps $\tau_\eps$ and $\sigma_\eps$ are related by
  \begin{equation*}
    \tau_\eps(\omega)_t=\left(\db\theta_x\right)^{-1}
    \left(\eps^{-1/2}\delta_\eps \left(\db\theta_x\left(
          \sigma_\eps(\omega)_t
        \right)\right)\right)
  \end{equation*}
  and because $\eps^{-1/2}\delta_\eps(y)$ tends to
  $(y^1,\dots,y^{d_1},0,\dots,0)$ as $\eps\to 0$,
  it follows that the $\sqrt{\eps}$-rescaled diffusion loop measures
  $\mu_\eps^{x,x}\circ\tau_\eps^{-1}$ converge weakly to
  $\tilde\mu^{x,x}\circ \psi^{-1}$ on $T\Omega^{0,0}$ as $\eps\to
  0$. Provided $\lo$ is non-elliptic at $x$,
  the latter is a degenerate measure which is
  supported on the set of paths $(v_t)_{t\in[0,1]}$ in $T\Omega^{0,0}$
  which satisfy $v_t\in C_1(x)$, for all $t\in[0,1]$.
  Hence, the rescaled diffusion process
  $(\eps^{-1/2}\theta(x_t^\eps))_{t\in[0,1]}$
  conditioned by $\theta(x_1^\eps)=0$ localises around the subspace
  $(\theta_*C_1)(0)$.

  Finally, by considering the limiting diffusion loop
  from Example~\ref{example2},
  we demonstrate that the degenerate limiting measure
  $\tilde\mu^{x,x}\circ \psi^{-1}$ need not be Gaussian.
  Going back to Example~\ref{example2}, we first observe that the map
  $\psi$ is simply projection onto the first
  component, i.e.
  \begin{equation*}
    \psi(v)_t=
    \begin{pmatrix}
      1 & 0 \\
      0 & 0
    \end{pmatrix} v_t\;.
  \end{equation*}
  Thus, to show that the measure $\tilde\mu^{x,x}\circ \psi^{-1}$ is
  not Gaussian,
  we have to analyse the process
  $(B_t^1,-\int_0^t(B_s^1)^2\std B_s^2)_{t\in[0,1]}$
  conditioned to return to $0$ at time $1$ and show that its
  first component is not Gaussian.
  Using the tower property, we first condition on
  $B_1^1=0$ to see that this component is equal in law to the
  process $(B_t^1-tB_1^1)_{t\in[0,1]}$ conditioned by
  $\int_0^1(B_s^1-sB_1^1)^2\std B_s^2=0$,
  where the latter is in fact equivalent to
  conditioning on $\int_0^1(B_s^1-sB_1^1)^2\dd B_s^2=0$. Let $\mu_{B}$
  denote the Brownian bridge measure on
  $\Omega(\R)^{0,0}=\{\omega\in
  C([0,1],\R)\colon\omega_0=\omega_1=0\}$
  and let $\nu$ be the law of $-\int_0^1(B_s^1-sB_1^1)^2\dd B_s^2$
  on $\R$. Furthermore, denote the joint law of
  \begin{equation*}
    \left(B_t^1-tB_1^1\right)_{t\in[0,1]}
    \qquad\mbox{and}\qquad
    -\int_0^1\left(B_s^1-sB_1^1\right)^2\dd B_s^2
  \end{equation*}
  on $\Omega(\R)^{0,0}\times\R$ by $\mu$.
  Since $-\int_0^1\omega_s^2\dd B_s^2$, for
  $\omega\in\Omega(\R)^{0,0}$ fixed, is a normal random variable with
  mean zero and variance $\int_0^1\omega_s^4\dd s$, we obtain that
  \begin{equation}\label{split}
    \mu(\db\omega,\db y)=
    \frac{1}{\sqrt{2\pi}\sigma(\omega)}
    \e^{-\frac{y^2}{2\sigma^2(\omega)}}\mu_B(\db\omega)\dd y
    \quad\mbox{with}\quad
    \sigma(\omega)=\left(\int_0^1\omega_s^4\dd s\right)^{1/2}\;.
  \end{equation}
  On the other hand, we can disintegrate $\mu$ as
  \begin{equation*}
    \mu(\db\omega,\db y)=\mu_B^y(\db\omega)\nu(\db y)\;,
  \end{equation*}
  where $\mu_B^y$ is the law of
  $(B_t^1-tB_1^1)_{t\in[0,1]}$ conditioned by
  $-\int_0^1(B_s^1-sB_1^1)^2\dd B_s^2=y$, i.e. we are interested in
  the measure $\mu_B^0$. From (\ref{split}), it follows that
  \begin{equation*}
    \mu_B^0(\db\omega)\propto\sigma^{-1}(\omega)\mu_B(\db\omega)=
    \left(\int_0^1\omega_s^4\dd s\right)^{-1/2}\mu_B(\db\omega)\;.
  \end{equation*}
  This shows that $\mu_B^0$ is not Gaussian, which implies that the
  $\sqrt{\eps}$-rescaled fluctuations indeed admit
  a non-Gaussian limiting
  diffusion loop.
\end{remark}

%%%%%%%%%%%%%%%%%%%%%%%%%%%%%%%%%%%%%%%%%%%%%%%%%%%%%%%%%%%%%%%%%%%%%%%%%%%%%%%
\bibliographystyle{plain}
\bibliography{books}

\end{document}